\documentclass[11pt,a4paper]{amsart}

\usepackage{titletoc}
\usepackage{titlesec}
\titleformat{\section}[block]{\Large\bfseries}{\thesection}{1em}{}
\titleformat{\subsection}[block]{\bfseries}{\thesubsection}{1em}{}
\dottedcontents{section}[1.5cm]{}{3em}{1pc}
\dottedcontents{subsection}[1.5cm]{}{3em}{1pc}

\usepackage{amsmath, amsthm, amssymb}
\usepackage{amsfonts}
\usepackage[utf8]{inputenc}
\usepackage[dvips]{epsfig}
\usepackage{graphicx}
\usepackage[english]{babel}
\usepackage{hyperref}
\hypersetup{colorlinks=true,linkcolor=black,citecolor=black}

\usepackage{mathptmx}
\usepackage{cite}
\usepackage{graphicx}

\allowdisplaybreaks
\usepackage{color}
\usepackage[dvipsnames]{xcolor}

\def\sideremark#1{\ifvmode\leavevmode\fi\vadjust{\vbox to0pt{\vss
 \hbox to 0pt{\hskip\hsize\hskip1em
 \vbox{\hsize2.1cm\tiny\raggedright\pretolerance10000
  \noindent #1\hfill}\hss}\vbox to15pt{\vfil}\vss}}}%

\let\oldsqrt\sqrt
\def\sqrt{\mathpalette\DHLhksqrt}
\def\DHLhksqrt#1#2{%
\setbox0=\hbox{$#1\oldsqrt{#2\,}$}\dimen0=\ht0
\advance\dimen0-0.2\ht0
\setbox2=\hbox{\vrule height\ht0 depth -\dimen0}%
{\box0\lower0.4pt\box2}}

\newcommand{\R}{\mathbb{R}} 
\newcommand{\N}{\mathbb{N}} 

\newcommand{\sOmega}{\text{\tiny $\Omega$}}
\newcommand{\cEO}{{{E}_{\sOmega}\,}}
\newcommand{\cEOprime}{{ E_{\sOmega'}\,}}
\newcommand{\cEOn}{{ E_{\sOmega_n}\,}}
\newcommand{\cRO}{{ R_{\sOmega}\,}}

\newcommand{\deltaO}{{\delta_{\sOmega}\!}}

\newcommand{\loglap}{L_{\text{\tiny $\Delta \,$}}}

\newcommand{\id}{\textnormal{id}} 
\newcommand{\dist}{\textnormal{dist}} 
\newcommand{\diam}{\textnormal{diam}} 
\newcommand{\supp}{\textnormal{supp}} 

\renewcommand{\phi}{\varphi}

\newcommand{\cE}{{\mathcal E}}

\newcommand{\cH}{{\mathcal H}}

\newcommand{\cGs}{{\mathbb G}}
\newcommand{\cHs}{{\mathbb H}}
\newcommand{\cQs}{{\mathbb Q}}
\newcommand{\cFs}{{\mathbb F}}
\newcommand{\cDs}{{\mathbb D}}

\newcommand{\eps}{\varepsilon}

\theoremstyle{definition}
\newtheorem{defi}{Definition}[section]
\newtheorem{remark}[defi]{Remark}

\theoremstyle{plain} 
\newtheorem{thm}[defi]{Theorem}
\newtheorem{prop}[defi]{Proposition}
\newtheorem{lemma}[defi]{Lemma}
\newtheorem{cor}[defi]{Corollary}

\theoremstyle{definition}

\numberwithin{equation}{section}

\title{A new look at the fractional Poisson problem via the Logarithmic Laplacian}

\author[Sven Jarohs]
{Sven Jarohs}
\address{Institut f\"ur Mathematik,
Goethe-Universit\"at Frankfurt.
Robert-Mayer-Str. 10
D-60629 Frankfurt am Main, Germany}
\email{jarohs@math.uni-frankfurt.de}

\author[Alberto Salda\~{n}a]
{Alberto Salda\~{n}a}
\address{Instituto de Matemáticas\\
Universidad Nacional Autónoma de México\\
Circuito Exterior, Ciudad Universitaria\\
04510 Coyoacán, Ciudad de México, Mexico\\
}
\email{alberto.saldana@im.unam.mx}

\author[Tobias Weth]
{Tobias Weth}
\address{Institut f\"ur Mathematik,
Goethe-Universit\"at Frankfurt.
Robert-Mayer-Str. 10
D-60629 Frankfurt am Main, Germany}
\email{weth@math.uni-frankfurt.de}

\date{\today}

\begin{document}
\begin{abstract}
We analyze the $s$-dependence of solutions $u_s$ to the family of fractional Poisson problems 
$$
(-\Delta)^s u =f \quad \text{in $\Omega$},\qquad u \equiv 0 \quad \text{on $\R^N\setminus \Omega$}
$$
in an open bounded set $\Omega \subset \R^N$, $s \in (0,1)$. In the case where $\Omega$ is of class $C^2$ and $f \in C^{\alpha}(\overline{\Omega})$ for some $\alpha>0$, we show that the map $(0,1) \to L^\infty(\Omega)$, $s\mapsto u_s$ is of class $C^1$, and we characterize the derivative $\partial_s u_s$ in terms of the logarithmic Laplacian of $f$. As a corollary, we derive pointwise monotonicity properties of the solution map $s \mapsto u_s$ under suitable assumptions on $f$ and $\Omega$. Moreover, we derive explicit bounds for the corresponding Green operator on arbitrary bounded domains which are new even for the case $s=1$, i.e., for the local Dirichlet problem $-\Delta u = f$ in $\Omega$, $u \equiv 0$ on $\partial \Omega$.
\end{abstract}
\maketitle

\begin{center}
	{CONTENTS}
\end{center}
\startcontents[sections]
\printcontents[sections]{l}{1}{\setcounter{tocdepth}{1}}

\section{Introduction}
The present paper is devoted to the fractional Poisson problem 
\begin{equation}\label{eq:prob1-poisson-fractional}
\left\{\begin{aligned}
(-\Delta)^s u&=f &&\text{ in $\Omega$,}\\
u&=0 &&\text{ on $\R^N  \setminus \Omega$}
\end{aligned}\right.
\end{equation}
for $s \in (0,1)$ in an open bounded set $\Omega \subset \R^N$, $N \ge 2$. For $f \in L^2(\Omega)$, (\ref{eq:prob1-torsion-fractional}) admits a unique weak solution $u \in \cH^s_0(\Omega)$. Here and in the following, we let $\cH^s_0(\Omega)$ be the completion of $C^\infty_c(\Omega)$ with respect to the norm $\|\cdot\|_{\cH^s}$ induced by the scalar product 
\begin{equation}
  \label{eq:scalar-product}
(w,v) \mapsto \cE_s(w,v) := \int_{\R^N} |\xi|^{2s} \hat w (\xi) \hat v(\xi)\,d\xi.
\end{equation}
Moreover, by definition, $u$ is a weak solution of (\ref{eq:prob1-poisson-fractional}) if 
\begin{equation}
\label{eq:prob1-poisson-fractional-weak}
 \cE_s(u,v) = \int_{\Omega}f v \,dx \qquad \text{for all $v \in \cH^s_0(\Omega)$.}
\end{equation}
These notions extend to the case $s=1$, in which 
(\ref{eq:prob1-poisson-fractional}) is replaced by the classical Poisson problem\begin{equation}
\label{eq:prob1-poisson-classical}
-\Delta u=f \quad \text{in $\Omega$,}\qquad \qquad 
u=0 \quad \text{ on $\partial \Omega$.}
\end{equation}
We also note that, if $\Omega$ has a continuous boundary, then $\cH^s_0(\Omega)$ coincides with the space of functions $w \in L^2(\R^N)$ with $w \equiv 0$ on $\R^N \setminus \Omega$ and $\cE_s(w,w)<\infty$, see \cite[Theorem 1.4.2.2]{G11}. Several important estimates have been proved in recent years regarding the unique weak solution of \eqref{eq:prob1-poisson-fractional}, which we denote by $\cGs_s f: \R^N \to \R$ in the following. In particular, if $\Omega$ is a bounded Lipschitz domain satisfying the exterior sphere condition, Ros-Oton and Serra have shown in \cite{RS12} that $\cGs_sf\in C^s(\R^N)$ for $s \in (0,1)$ if $f\in L^{\infty}(\Omega)$, thus giving the optimal regularity up to the boundary. Moreover, Silvestre \cite{S07}, Grubb \cite{G15:2} and Ros-Oton and Serra \cite{RS12,RS16} provide, in particular, estimates on the interior regularity. For $f\in C^\alpha(\overline{\Omega})$, it follows that $\cGs_sf$ is  the unique classical bounded solution of \eqref{eq:prob1-poisson-fractional}. For more information on the fractional Laplacian and its applications, see \cite{BV16} and the references therein.

\medskip

The main aim of this paper is to study the dependence of $\cGs_s f$ on the parameter $s \in (0,1)$ on general bounded domains. In particular, we shall give answers to the following questions.
{\em
\begin{enumerate}
	\item[(Q1)] Under which assumptions on $f$ and $\Omega$ is the map $s \mapsto \cGs_s f$ differentiable in a suitable function space, and how can we characterize its derivative?
\item[(Q2)] Under which assumptions on $f$ and $\Omega$ is this map pointwisely decreasing in $\Omega$?
\end{enumerate}}
As a byproduct of our results, we derive new bounds for the operator norm $\|\cGs_s\|$ of $\cGs_s$ with respect to $L^\infty(\Omega)$, which is defined as the smallest constant $C=C(N,\Omega,s)$ with the property that
\begin{equation}
  \label{eq:starting-ineq}
\|\cGs_s f\|_{L^{\infty}(\Omega)}\leq C\|f\|_{L^{\infty}(\Omega)} \qquad \text{for all $f \in L^\infty(\Omega)$.}
\end{equation}
Since $\cGs_s$ is order preserving as a consequence of the maximum principle for $(-\Delta)^s$, it follows that
\begin{align*}
\|\cGs_s\|= \sup_{\Omega} [\cGs_s 1]. 
\end{align*}
In other words, $\|\cGs_s\|$ is given as the maximal value of the (unique) solution to the {\em fractional torsion problem} 
\begin{equation}\label{eq:prob1-torsion-fractional}
\left\{\begin{aligned}
(-\Delta)^s u&=1 &&\text{ in $\Omega$,}\\
u&=0 &&\text{ on $\R^N  \setminus \Omega$}
\end{aligned}\right.
\end{equation}
In our results, we shall thus pay special attention to this particular problem.
As an example motivating our study, let us consider the case $\Omega=B_r(0)$ for fixed $r>0$. In this case the solution $u_s:= \cGs_s 1$ of (\ref{eq:prob1-torsion-fractional}) is given by 
\begin{equation}
  \label{eq:explicit-formula-fractional-torsion}
u_s(x)=\gamma_{N,s}(r^2-|x|^2)_+^{s}, \quad \text{with $\gamma_{N,s}=\frac{\Gamma(\frac{N}{2})}{4^s\Gamma(s+1)\Gamma(\frac{N}{2}+s)}$,}
\end{equation}
see e.g. \cite{G61,D12}. We may thus compute the pointwise derivative $v_s=\partial_s u_s$ with respect to $s$ as follows: 
\[
v_s(x)=u_{s}(x)\Bigl[ \ln\bigl (r^2-|x|^2)-\bigl(2\ln(2)+\psi(\frac{N}{2}+s)+\psi(s+1)\bigr)\Bigr] \qquad \text{for $x \in \Omega$.}
\]
Here $\psi:= \frac{\Gamma'}{\Gamma}$ denotes the digamma function. Clearly, $v_s$ is nonpositive in $B_r(0)$ if and only if $2\ln(r)\leq 2\ln(2)+\psi(\frac{N}{2}+s)+\psi(s+1)$. Since the digamma function $\psi$ is increasing on $(0,\infty)$ (see e.g. \cite[eq. 6.3.21]{AS64}), we find that $v_s$ is nonpositive in $B_r(0)$ for all $s \in (0,1)$ if and only if 
$2\ln(r)\leq 2\ln(2) +\psi(\frac{N}{2})-\gamma$, i.e., 
\begin{equation}
  \label{eq:necessary-sufficient}
r \le r_N:= 2 e^{\frac{1}{2}[\psi(\frac{N}{2})- \gamma]},
\end{equation}
where $\gamma:= -\psi(1)=-\Gamma'(1)$ is the Euler-Mascheroni constant. Consequently, $u_s$ is pointwisely decreasing in $s \in (0,1)$ on $\Omega=B_r(0)$ if 
and only if $r \le r_N$. 

Our aim is to derive monotonicity properties and rate of change formulas for more general bounded domains $\Omega$ and source functions $f$ where no explicit form of $u_s = \cGs_s f$ is available. Essential in this analysis is the \textit{logarithmic Laplacian} $\loglap$ introduced in \cite{CW18}, which is a weakly singular Fourier integral operator associated to the symbol $2\ln|\cdot|$. The operator $\loglap$ can be seen as the derivative of $(-\Delta)^s$ at $s=0$. More precisely, as shown in \cite[Theorem 1.1]{CW18}, if $\phi\in C^{\alpha}_c(\R^N)$ for some $\alpha>0$, then $\frac{d}{ds}\Big|_{s=0}(-\Delta)^s\phi=\loglap \phi$
in $L^p(\R^N)$ for $1< p \le \infty$. Moreover, $\loglap \phi$ admits the integral representation
\begin{equation}\label{eq:log-laplace}
\loglap \phi(x)=c_NP.V.\int_{B_1(0)}\frac{\phi(x)-\phi(x+y)}{|y|^{N}}\ dy - c_N\int_{\R^N\setminus B_1(0)}\frac{\phi(x+y)}{|y|^{N}}\ dy+\rho_N\phi(x),  
\end{equation}
where 
\begin{equation}
  \label{eq:def-c-n-rho-n}
c_N=\frac{\Gamma(\frac{N}{2})}{\pi^{N/2}}= \frac{2}{|S^{N-1}|} \qquad \text{and}\qquad \rho_N :=2\ln(2)+\psi(\frac{N}{2}) - \gamma.   
\end{equation}  
Here and in the following, $|S^{N-1}|$ denotes the $N-1$-dimensional volume of the unit sphere in $\R^N$. Moreover, as noted in \cite[Proposition 1.3]{CW18}, the value of $\loglap \phi(x)$ is also well defined by~(\ref{eq:log-laplace}) if $\phi$ is merely Dini continuous in $x$ and satisfies $\int_{\R^N}(1+|z|)^{-N}|\phi(z)|dz < \infty$.  We also use the following alternative integral representation for $\loglap$ with respect to an open subset $\Omega$ of $\R^N$ given by 
\begin{equation}\label{eq:log-laplace2}
\loglap \phi(x)=c_N\int_{\Omega}\frac{\phi(x)-\phi(y)}{|x-y|^N}\ dy-c_N\int_{\R^N\setminus \Omega}\frac{\phi(y)}{|x-y|^N}\ dy +[h_{\Omega}(x)+\rho_N]\phi(x),\quad x\in \Omega,
\end{equation}
with the function 
\begin{equation}\label{eq:h-func}
h_\Omega: \Omega \to \R, \qquad h_{\Omega}(x)=c_N \int_{B_1(x)\setminus \Omega}\frac{1}{|x-y|^N}\ dy - c_N \int_{\Omega\setminus B_1(x)}\frac{1}{|x-y|^{N}}\ dy,
\end{equation}	
see \cite[Proposition 2.2]{CW18}. As a consequence of (\ref{eq:log-laplace2}), we have 
$$
[\loglap 1_\Omega](x) = h_{\Omega}(x)+\rho_N \qquad \text{for $x \in \Omega$.}
$$
In the particular case $\Omega= B_r(0)$, $r>0$, it was proved in \cite[Lemma 4.11]{CW18} that $\inf \limits_{x \in \Omega}h_{\Omega}(x)= h_\Omega(0)=-2 \ln r$. Therefore, the condition (\ref{eq:necessary-sufficient}) can be rephrased in the form 
$$
\loglap 1_{B_r(0)} \ge 0 \quad \text{in $B_r(0)$.}
$$
It is therefore tempting to guess that ---in general open bounded sets $\Omega$--- the nonnegativity of the function 
$$
\loglap 1_\Omega \equiv  h_\Omega + \rho_N
$$
in $\Omega$ is still a necessary and sufficient condition for the solution map $s \mapsto \cGs_s 1$ to be pointwisely decreasing in $s$ on $[0,1)$. We shall give a positive answer to this question in one of our main results. More generally, we shall see that, if $f \in C^\alpha(\overline \Omega)$ is a nonnegative function for some $\alpha>0$ and $\cEO f$ denotes the trivial extension of $f$ to $\R^N$, then the nonnegativity of $\loglap \cEO f$ in $\Omega$ is a necessary and sufficient condition for the solution map $s \mapsto \cGs_s f$ to be pointwisely decreasing in $s$ on $[0,1]$. The first main result of this paper provides a derivative of the map $s \mapsto \cGs_s f$ in a suitable sense. 

\begin{thm}\label{thm1:differential}
Let $N\geq 2$, let $\Omega\subset\R^N$ be an open and bounded set of class $C^2$, and let $f\in C^{\alpha}(\overline{\Omega})$ for some $\alpha>0$. Then we have: 
\begin{enumerate}
\item[(i)] The map
$$
(0,1) \to L^\infty(\Omega), \qquad s \mapsto u_s:= \cGs_s f
$$
is of class $C^1$. Moreover, for every $s \in (0,1)$, the function $v_s:= \partial_s u_s \in L^\infty(\Omega)$ is given as the unique solution of the boundary value problem 
\begin{equation}\label{thm1:eq2}
\left\{
  \begin{aligned}
(-\Delta)^s v_s &= -\loglap \bigl(\cEO f + w_s\bigr)&& \qquad \text{in $\Omega$,}\\
v_s &= 0 && \qquad \text{in $\R^N \setminus \Omega$,}  
  \end{aligned}
\right.
\end{equation}
where $\cEO f$ denotes the trivial extension of $f$ to $\R^N$ and $w_s=[(-\Delta)^s u_s]1_{\R^N \setminus \overline{\Omega}}$, i.e., 
\begin{equation}
  \label{eq:def-w-s}
w_s(x) = \left\{
  \begin{aligned}
&0,&& \qquad \text{$x \in \overline \Omega$,}\\
&[(-\Delta)^s u_s](x)= -c_{N,s} \int_{\Omega}\frac{u_s(y)}{|x-y|^{N+2s}}dy,&& \qquad \text{$x \in \R^N \setminus \overline{\Omega}$.}    \end{aligned}
\right.
\end{equation}
\item[(ii)] If, in addition, $f \geq 0$ in $\Omega$ and $f\not\equiv0$, then $v_s < 0$ in $\Omega$ for all $s \in (0,1)$ if and only if  
  \begin{equation}
    \label{eq:loglap-cond}
\loglap \cEO f \ge 0 \qquad \text{in $\Omega$.}
  \end{equation}
\end{enumerate}
\end{thm}

Before presenting applications of this result, we briefly comment on its proof. The main difficulty is given by the fact that the family of Green operators $\cGs_s f$, $s \in [0,1]$ does not obey the semigroup property $\cGs_{s+t} = \cGs_s \cGs_t$, so the differentiablity of the solution map $s \mapsto \cGs_s f$ at $s \in (0,1)$ cannot be reduced to a consideration of the case $s=0$. In particular, in the case $s>0$ we are led to also consider variants of problem \eqref{eq:prob1-poisson-fractional} with nonhomogeneous Dirichlet conditions on $\R^N\setminus \Omega$ and corresponding representation formulas for solutions via Poisson kernels.  Starting with the work of Bogdan \cite{B99}, these formulas have been derived under rather general assumptions on $\Omega$, see \cite{BKK08,nicola} and the references therein. In Section \ref{sec:solution-map-related} below, we recall these formulas as we need them in our analysis.

Under the same assumptions as in Theorem~\ref{thm1:differential}, we shall also derive the limiting properties 
$u_s \to f$ and $\frac{u_s-f}{s} \to \loglap \cEO f$ in $\Omega$ as $s \to 0^+$. However, we cannot expect convergence in $L^\infty(\Omega)$ since $u_s$ vanishes on $\partial \Omega$ for $s>0$ and $f$ does not in general. We consider instead the following type of convergence. 

\begin{defi}
\label{defi-almost-uniform-convergence}
Let $\Omega \subset \R^N$ be an open bounded set, and let $(u_n)_n$ be a sequence of real-valued functions defined on $\Omega$.  We say that $(u_n)_n$ converges  \emph{almost uniformly in $\Omega$} to a function $u: \Omega \to \R$ if 
$$
\lim_{n \to \infty}\|\deltaO^{\eps} (u_n-u) \|_{L^\infty(\Omega)}=0\qquad \text{for every $\eps>0$.}
$$
Here and in the following, $\deltaO: \R^N \to \R$ is the boundary distance function, i.e., $\deltaO(x)= \dist(x,\partial \Omega)$ for $x \in \R^N$.
\end{defi}

We note that, if $u_n \to u$ almost uniformly in a bounded open set $\Omega$, then $u_n \to u$ in $L^\infty_{loc}(\Omega)$. If moreover $\partial \Omega$ is Lipschitz, then also $u_n \to u$ in $L^p(\Omega)$ for every $p \in [1,\infty)$. 

\begin{thm}\label{thm1:differential-s-0}
Let $N\geq 2$, let $\Omega\subset\R^N$ be an open and bounded set of class $C^2$, let $f\in C^{\alpha}(\overline{\Omega})$ for some $\alpha>0$, and let $u_s:= \cGs_s f$ for $s \in (0,1)$. Then
$$
u_s \to f  \quad \text{and}\quad \quad \frac{u_s-f}{s} \to \loglap \cEO f \qquad \text{almost uniformly in $\Omega$ as $s\to 0^+$.}
$$
Moreover, if $f \geq 0$ in $\Omega$, then
\begin{equation}
  \label{eq:integral-derivative-estimate}
0 \le u_s(x) \le f(x) - \int_0^s \bigl(\cGs_t[\loglap (\cEO f)]\bigr)(x)\,dt \qquad \text{for $s \in (0,1)$, $x \in \Omega$.}
\end{equation}
\end{thm}

We remark that the estimate (\ref{eq:integral-derivative-estimate}) is a consequence of Theorem~\ref{thm1:differential} and the first part of Theorem~\ref{thm1:differential-s-0}. With the help of this estimate, we will be able to give the following answer to Question (Q2) raised above.
 
\begin{thm}
\label{pointwise-decreasing}
Let $N\geq 2$, let $\Omega\subset\R^N$ be an open and bounded set, and let $f\in C^{\alpha}(\overline{\Omega})$ for some $\alpha>0$ with $f  \ge 0$ in $\Omega$. Then (\ref{eq:loglap-cond}) is a sufficient condition for $u_s:= \cGs_s f$ to be pointwisely decreasing in $\Omega$ with respect to $s \in [0,1]$. If, moreover, $\partial \Omega$ is of class $C^2$, then this condition is also necessary. 
\end{thm}

Here and in the following, for matters of consistency, we set $\cGs_0 f = \cEO f$ for $f \in L^\infty(\Omega)$. 

Theorem~\ref{pointwise-decreasing} follows rather directly from Theorem~\ref{thm1:differential-s-0} in the case where $\partial \Omega$ is of class $C^2$. By approximation, we shall also show the statement in the case of arbitrary open bounded sets $\Omega \subset \R^N$. In our next observation, we remark that condition (\ref{eq:loglap-cond}) is inherited via domain inclusion. 

\begin{lemma}
\label{domain-inclusion}
Let $N\geq 2$, let $\Omega, \Omega' \subset\R^N$ be open and bounded sets with $\Omega' \subset \Omega$, and let $f\in C^{\alpha}(\overline{\Omega})$ for some $\alpha>0$. Then we have 
$$
[\loglap \cEOprime f - \loglap \cEO f](x) = c_N \int_{\Omega \setminus \Omega'}\frac{f(y)}{|x-y|^N}dy \qquad \text{for $x \in \Omega'$,}
$$
Here we identify $f$ with the restriction of $f$ to $\Omega'$. In particular, if $f \ge 0$ in $\Omega$ and $\loglap \cEO f \ge 0$ in $\Omega$, then also $\loglap \cEOprime f \ge 0$ in $\Omega'$.  
\end{lemma}
 
Let us put these results into perspective by discussing the case $f \equiv 1$, i.e. the fractional torsion problem (\ref{eq:prob1-torsion-fractional}). As noted already, in this case it follows from (\ref{eq:log-laplace2}) that 
$$
\loglap \cEO f = \loglap 1_\Omega = h_\Omega + \rho_N \qquad \text{in $\Omega$.}
$$
Moreover, $\loglap 1_{B_r(0)}$ is nonnegative in $B_r(0)$ if and only if $r \le r_N$, where $r_N$ is defined in (\ref{eq:necessary-sufficient}). Consequently, if $\Omega$ is an open subset of $B_{r_N}(0)$, then $\loglap 1_\Omega \ge 0$ in $\Omega$ by Lemma~\ref{domain-inclusion}. Together with Theorem~\ref{pointwise-decreasing}, this observation gives rise to the following corollary.

\begin{cor}
\label{inclusion-rNball}
Let $N \ge 2$, $r_N:= 2 e^{\frac{1}{2}[\psi(\frac{N}{2})- \gamma]}$, and let $\Omega \subset B_{r_N}(0)$ be an open set. Then $u_s:= \cGs_s 1$ is pointwisely decreasing in $\Omega$ with respect to $s \in [0,1]$.  
\end{cor}
 
We remark that this corollary is optimal in a certain sense. More precisely, if $\Omega \subset \R^N$ is an open bounded set with $B_r(0) \subset \Omega$ for some $r>r_N$, then we have $[\cGs_s 1](0)>1$ for $s>0$ sufficiently small by (\ref{eq:explicit-formula-fractional-torsion}) and the maximum principle. Thus the map $s \mapsto [\cGs_s 1](0)$ cannot be decreasing on $(0,1)$.  

In order to derive bounds for the $L^\infty$-operator norm of the Green operator $\cGs$ in an arbitrary open and bounded subset of $\R^N$, we set 
$$
h_0(\Omega) = \inf_{x \in \Omega}h_\Omega(x).
$$
Since the operator $\cGs_s$ is order preserving for $s \in (0,1]$, we find that $$
\cGs_s(\loglap 1_\Omega)= \cGs_s(h_\Omega+\rho_N) \ge (h_{0}(\Omega)+\rho_N) \cGs_s 1 = (h_{0}(\Omega)+\rho_N) u_s \qquad \text{in $\Omega$.}
$$
If $\partial \Omega$ is of class $C^2$, we may combine this estimate with (\ref{eq:integral-derivative-estimate}) to obtain 
$$
u_s(x) \le 1 - (h_0(\Omega)+\rho_N)\int_0^s u_t(x)\,dt \qquad \text{for $x \in \Omega$, $s \in (0,1)$}
$$
and therefore, by Grönwall's inequality,
$$
u_s \le e^{-s \bigl(h_0(\Omega)+\rho_N\bigr)} \qquad \text{in $\Omega$ for $s \in (0,1)$.}
$$
By approximation, we will obtain this bound also for arbitrary open bounded sets $\Omega \subset \R^N$. Recalling that $\|\cGs_s\|= \sup_{\Omega} u_s$, we therefore infer the following estimate.

\begin{cor}\label{cor:application-fractional}
Let $N \ge 2$, $\Omega\subset \R^N$ be an arbitrary open and bounded set. Then we have 
\begin{align*}
\|\cGs_s\| \leq e^{-s \bigl(h_0(\Omega)+\rho_N\bigr)}\quad \text{ for $s \in [0,1]$.}
\end{align*}

\end{cor}

The geometry of $\Omega$ enters in the estimate for $\|\cGs_s\|$ via the quantity $h_0(\Omega)$. To deduce a less sharp but more intuitive bound, we define, for $r>0$, the {\em relative $r$-density} of the set $\Omega \subset \R^N$ by 
$$
d_r(\Omega) = \sup_{x \in \Omega}d_r(x,\Omega) \qquad \text{with}\quad d_r(x,\Omega)= \frac{|B_r(x) \cap \Omega|}{|B_r|}.
$$
We shall see in Section~\ref{sec:compl-proofs-main} below that 
\begin{equation}
  \label{eq:h-0-lower-bound-rel-density}
h_0(\Omega) \ge - \frac{2}{N} \ln \Bigl(\bigl(\frac{|\Omega|}{|B_1|}+r^N[1-d_r(\Omega)]\bigr)d_r(\Omega)\Bigr)
\end{equation}
for any open and bounded $\Omega \subset \R^N$ and every $r>0$.
Therefore, we deduce the following estimate from Corollary \ref{cor:application-fractional}.

\begin{thm}\label{cor:application-fractional-relative-density}
	Let $N \ge 2$, $r>0$, and let $\Omega\subset \R^N$ be an arbitrary open and bounded set. Then
$$
\|\cGs_s\| \le e^{-s\rho_N}d_r(\Omega)^{\frac{2s}{N}} \Bigl(\frac{|\Omega|}{|B_1|}+ r^{N}(1 -d_r(\Omega))\Bigr)^{\frac{2s}{N}} \qquad \text{for $s \in [0,1]$.}
$$
\end{thm}

We point out that this result is new even in the case $s=1$ and yields the bound
\begin{equation}
  \label{eq:new-G-1-bound}
\|\cGs_1\| \le e^{-\rho_N}d_r(\Omega)^{\frac{2}{N}} \Bigl(\frac{|\Omega|}{|B_1|}+ r^{N}(1 -d_r(\Omega))\Bigr)^{\frac{2}{N}}. 
\end{equation}
This estimate should be compared with the classical bound 
\begin{equation}
  \label{talenti-bound}
\|\cGs_1\|\le \frac{1}{2N}\bigl(\frac{|\Omega|}{|B_1|} \bigr)^{\frac{2}{N}}. 
\end{equation}
which arises from Talenti's comparison principle \cite{T76}. The bound (\ref{talenti-bound}) is sharp among domains with fixed volume $|\Omega|$ and is attained on balls in $\R^N$. On the other hand, (\ref{eq:new-G-1-bound}) is better than (\ref{talenti-bound}) for domains which are ``thin'' compared to its total volume, i.e., for domains with a small relative $r$-density for some $r>0$. 

We remark that bounds for $\|\cGs_s\|$ are important when considering 
solutions $u \in L^\infty(\Omega)$ of semilinear problems of the type 
\begin{equation}
  \label{eq:semilinear-fractional}
(-\Delta)^s u = f(u) \quad \text{in $\Omega$},\qquad u = 0 \quad \text{on $\R^N \setminus \Omega$}
\end{equation}

for $s \in (0,1)$ and 
\begin{equation}
  \label{eq:semilinear-classical}
-\Delta u = f(u) \quad \text{in $\Omega$},\qquad u = 0 \quad \text{on $\partial \Omega$}
\end{equation}
with a Lipschitz nonlinearity $f: \R \to \R$. These solution correspond to fixed points of the nonlinear operator 
$$
T_s: L^\infty(\Omega) \to L^\infty(\Omega), \qquad T_s u = \cGs_s f(u).
$$
If $f$ admits a Lipschitz constant less than $\frac{1}{\|\cGs_s\|}$, then the operator $T_s$ is a contraction and therefore the problems (\ref{eq:semilinear-fractional}) and (\ref{eq:semilinear-classical}) admit a unique bounded solution.

\medskip

To close this introduction, we point out the natural question whether the differentiability of the map $s \mapsto \cGs_s f$ extends to values $s \ge 1$ under suitable assumptions on $f$ and $\Omega$. Higher order fractional Dirichlet problems and the associated Green operators $\cGs_s$, $s \ge 1$, have been studied e.g. in \cite{AJS17b,AJS16a,AJS16b,G15:2}. The main difficulty in this case is the fact that the function $w_s$ defined in (\ref{eq:def-w-s}) grows like 
$\deltaO^{-s}$ near the boundary $\partial \Omega$ and is therefore not contained in $L^1_{loc}(\R^N)$ if $s\ge 1$. We also remark that our main results are stated only for $N\geq 2$, although we expect that the statements also hold in the case $N=1$. Note that if $\Omega\subset \R$ is a bounded interval, then the corresponding Green operator has an explicit formula (see \cite{BGR61,B16,DG17,AJS16b}) and its derivative in $s$ can be computed directly for $s>0$. 


Finally, we mention the works \cite{SV17} resp. \cite{PV18}, where the $s$-dependence of boundary value problems has been considered for spectral-theoretic fractional powers of the Dirichlet and Neumann Laplacian, respectively. In particular, associated optimization problems in $s$ are studied. In this context we also mention \cite{BW20}, where the authors study an image regularization problem, which involves the variation of the parameter $s$ for the spectral-theoretic $s$-th fractional power of the Laplacian on the torus.

\medskip

The paper is organized as follows. Section \ref{sec:preliminary-estimates} is devoted to preliminary integral estimates which we shall use multiple times within our proofs. In Section \ref{sec:solution-map-related}, we collect various results on operators associated with the solution map $s \mapsto \cGs_s f$ for problem (\ref{eq:prob1-poisson-fractional}). The differentiability of the solution map $s \mapsto \cGs_s f$ at $s=0$ in suitable function spaces is then proved in Section \ref{sec:thecases0}. In Section \ref{sec:continuity} we discuss continuity properties of the solution map, and in Section \ref{sec:main-part} we prove its differentiability in the interval $s \in (0,1)$. In order to pass from domains of class $C^2$ to arbitrary bounded domains in some of our main results, we consider inner domain approximations and associated continuity properties in Section~\ref{sec:appr-nonsm-doma}. In Section~\ref{sec:compl-proofs-main},  we then complete the proofs of our main theorems. 

\subsection{Notation}\label{sec:notation}

For a set $A \subset\R^N$ and $x \in \R^N$, we define $\delta_A(x):=\dist(x,\partial A)$. Moreover, for given $r>0$, let $B_r(A)=\{x\in \R^N\;:\; \dist(x,A)<r\}$, and let $B_r(x)=B_r(\{x\})$ denote the ball of radius $r$ centered in $x\in \R^N$. If $x=0$ we also write $B_r$ instead of $B_r(0)$. Finally, we let $S^{k}=\{x\in \R^{k+1}\;:\; |x|=1\}$ denote the unit sphere in $\R^{k+1}$, $k \in \N$.

\section{Preliminary integral estimates}
\label{sec:preliminary-estimates}

This section is devoted to preliminary integral estimates which are of key importance in the proof of our main results. We start with two elementary estimates.

\begin{lemma}
\label{general-est}
We have 
$$
|\rho^{t}-\rho^{s}| \le \frac{t-s}{\eps}\bigl(\rho^{t+\eps}+\rho^{s-\eps}\bigr)\qquad \text{for $\rho,\eps>0$ and $s,t \in \R$, $s<t$.}
$$
\end{lemma}

\begin{proof}
We write 
$$
|\rho^{t}-\rho^{s}|= (t-s) |\ln \rho| \int_0^1 \rho^{s+ \tau (t-s)}d\tau. 
$$
If $\rho \ge 1$, we have $|\ln \rho| = \ln \rho \le \frac{\rho^\eps}{\eps}$,
and therefore 
$$
|\rho^{t}-\rho^{s}| \le \frac{t-s}{\eps}  \int_0^1 \rho^{s+ \tau (t-s)+\eps}d\tau \le \frac{t-s}{\eps} \rho^{t+\eps}.
$$
Moreover, for $\rho \in (0,1)$ we have $|\ln \rho| = - \ln \rho = \ln \rho^{-1} \le \frac{\rho^{-\eps}}{\eps}$ and therefore 
$$
|\rho^{t}-\rho^{s}| \le \frac{t-s}{\eps}  \int_0^1 \rho^{s+ \tau (t-s)-\eps}d\tau \le \frac{t-s}{\eps} \rho^{s-\eps}.
$$
\end{proof}

\begin{lemma}
\label{basic-estimate}
We have 
$$
\int_0^1 t^{-a} (c+ t)^{\lambda-1}\,dt \le \kappa(a,\lambda,c) \qquad \text{for $a,\lambda \in (-\infty,1)$ and $c \in (0,1]$,}
$$
where 
$$
\kappa(a,\lambda,c) :=
\left \{
  \begin{aligned}
&\min \{(\lambda-a)^{-1}, (1-a)^{-1}+|\ln c|\} &&\qquad \text{if $\lambda>a$,}\\
&c^{\lambda-a} \min\{ [(1-a)(a-\lambda)]^{-1}(1-\lambda),(1-a)^{-1} + |\ln c|\}&&\qquad \text{if $\lambda<a$,}\\
&(1-a)^{-1}+ |\ln c| && \qquad \text{if $\lambda=a$.}
  \end{aligned}
\right.
$$
\end{lemma}

\begin{proof}
Substituting $t =  c \tau^{\frac{1}{1-a}}$, we get 
$$
\int_0^1 t^{-a} (c+ t)^{\lambda-1}\,dt = \frac{c^{\lambda-a}}{1-a}\int_0^{c^{a-1}}(1+\tau^{\frac{1}{1-a}})^{\lambda-1}\,d\tau \le \frac{c^{\lambda-a}}{1-a}\Bigl(1 + \int_1^{c^{a-1}} \tau^{\frac{\lambda-1}{1-a}}\,d\tau\Bigr),
$$
where 
$$
\int_1^{c^{a-1}} \tau^{\frac{\lambda-1}{1-a}}\,d\tau = 
\left\{
\begin{aligned}
&\frac{1-a}{\lambda-a}(c^{a-\lambda}-1),&&\qquad \lambda \not = a,\\
&(1-a)|\ln c|,&&\qquad \lambda =a.
\end{aligned}
\right. 
$$
In particular,
$$
\int_0^1 t^{-a} (c+ t)^{\lambda-1}\,dt \le \frac{1}{1-a}  +|\ln c| 
\qquad \text{if $\lambda=a$.} 
$$
Moreover, if $\lambda>a$, 
\begin{align*}
\int_0^1 t^{-a} (c+ t)^{\lambda-1}\,dt &\le \frac{c^{\lambda-a}}{1-a}\Bigl(1 + \frac{1-a}{\lambda-a}(c^{a-\lambda}-1)\Bigr)
=\frac{1-a + (\lambda-1)c^{\lambda-a}}{(1-a)(\lambda-a)} \le \frac{1}{\lambda-a} 
\end{align*}
and also 
$$
\int_0^1 t^{-a} (c+ t)^{\lambda-1}\,dt \le \frac{c^{\lambda-a}}{1-a}\Bigl(1 + \frac{1-a}{\lambda-a}(c^{a-\lambda}-1)\Bigr)
\le \frac{c^{\lambda-a}}{1-a}\bigl(1 + (1-a)|\ln c|c^{a-\lambda}\bigr)\le \frac{1}{1-a} + |\ln c|.
$$
Finally, if $\lambda<a$, we have, similarly,
$$
\int_0^1 t^{-a} (c+ t)^{\lambda-1}\,dt \le  \frac{c^{\lambda-a}}{1-a} + \frac{c^{\lambda-a}-1}{a-\lambda}
\le c^{\lambda-a} \min \Bigl\{ \frac{(1-\lambda)}{(1-a)(a-\lambda)}, \frac{1}{1-a} + |\ln c| \Bigr\}.
$$
The claim follows.
\end{proof}

The next lemma, which is the main result of this section, is concerned with integrals involving singular kernels and powers of the function $\deltaO$ for a given open bounded set $\Omega$ of class $C^2$. 

\begin{lemma}\label{int:la}
	Let $N\geq 2$, $\Omega\subset \R^N$ be an open bounded set with $C^{2}$-boundary, $a\in(-1,1)$, and $\lambda \in [-N,1)$. Moreover, for $\delta>0$, we define 
$$
m(a,\lambda,\delta):=
\frac{1}{1-a}  \left\{
  \begin{aligned}
&\frac{1}{1-\lambda} \min \Bigl\{\frac{1}{\lambda-a},1+|\ln \delta|\Bigr\} &&\qquad \text{if $\lambda>a$,}\\
&1\:+\: \delta^{\lambda-a} \min \Bigl\{\frac{1}{a-\lambda},\frac{1}{1-\lambda}(1+|\ln \delta|)\Bigr\} &&\qquad \text{if $\lambda<a$,}\\
&\frac{1}{1-\lambda}\Bigl(1+|\ln \delta|\Bigr) && \qquad \text{if $\lambda=a$.}
  \end{aligned}
\right.
$$
Then there is $C=C(N,\Omega)>0$ with 
\begin{align}
  \label{eq:(i)-int-est}
	\int_{\Omega}\deltaO(y)^{-a}|x-y|^{\lambda-N}\ dy
	&\;\leq\; C\frac{m(a,\lambda,\deltaO(x))}{1+\deltaO(x)^{N-\lambda}} \qquad \text{for $x \in \R^N \setminus \overline \Omega\:$ and}\\
  \label{eq:(ii)-int-est}
	\int_{B_{r_\Omega}(x) \setminus \Omega}\deltaO(y)^{-a}|x-y|^{\lambda-N}\ dy
	&\;\leq\; C\: m(a,\lambda,\deltaO(x)) \qquad \text{for $x \in \Omega$,}
\end{align}
where $r_\Omega:= \diam \:\Omega+1$. 
\end{lemma}

\begin{proof}
We only prove (\ref{eq:(i)-int-est}), the proof of (\ref{eq:(ii)-int-est}) is very similar. In the following, the letter $C$ stands for positive constants depending only on $\Omega$ and the dimension $N$ but not on $a$ and $\lambda$. The value of $C$ may increase in every step of the estimate. 
Since $\Omega$ is of class $C^2$, there exist $\eps= \eps(\Omega) \in (0,1)$ and a diffeomorphism $T: \partial \Omega \times 
(-\eps,\eps) \to B_\eps(\partial \Omega)$ with bounded derivative and the property that 
\begin{equation}\label{int:la-propT}
T^{-1}(x)= \left \{
  \begin{aligned}
  &(p(x),\deltaO(x)),&& \qquad x \in B_\eps(\partial \Omega) \cap  \overline \Omega,\\ 
  &(p(x),-\deltaO(x)),&& \qquad x \in B_\eps(\partial \Omega) \setminus \overline \Omega, 
  \end{aligned}
\right.
\end{equation}
where $p(x)$ is the projection of $x \in B_\eps(\partial \Omega)$ onto $\partial \Omega$, i.e., the unique point in $\partial \Omega$ with $|x-p(x)|= \deltaO(x)$. In particular, this implies that 
$$
T(\partial \Omega \times 
(0,\eps)) = \Omega \cap B_\eps(\partial \Omega) \quad \text{and}\quad 
T(\partial \Omega \times 
(-\eps,0)) = B_\eps(\partial \Omega) \setminus \overline \Omega.
$$
Now, we first consider $x \in \R^N \setminus B_\eps(\Omega)$. Then 
$|x-y|^{\lambda-N} \le \frac{C}{1+\deltaO(x)^{N-\lambda}}$ for every $y \in \Omega$ and therefore 
\begin{equation}
\label{first-long-lemma-first-est-zero}
\int_{\Omega} \deltaO(y)^{-a} |x-y|^{\lambda-N}dy  \le 
\frac{C}{1+\deltaO(x)^{N-\lambda}} 
\int_\Omega \deltaO(y)^{-a}dy,
\end{equation}
where 
$$
\int_\Omega \deltaO(y)^{-a}dy \le \eps^{-a} |\Omega| + 
\int_{B_\eps(\partial \Omega) \cap \Omega}\deltaO(y)^{-a}dy\\
\le C + \int_{\partial \Omega \times [0,\eps]} t^{-a} T^*(dy)(z,t).
$$
Here $T^*(dy)$ denotes the pull-back of the volume element $dy$. Since $T$ has a bounded derivative, we conclude that 
\begin{equation}
  \label{eq:est-add-revision}
\int_\Omega \deltaO(y)^{-a}dy \le  C \Bigl(1+ |\partial \Omega| 
\int_{0}^\eps t^{-a}dt\Bigr)\le \frac{C}{1-a}
\end{equation}
and therefore, by (\ref{first-long-lemma-first-est-zero}), 
\begin{equation}
\label{first-long-lemma-first-est}
\int_{\Omega} \deltaO(y)^{-a} |x-y|^{\lambda-N}dy  \le \frac{C}{(1-a)\bigl(1+
\deltaO(x)^{N-\lambda}\bigr)}.
\end{equation}
Next, we consider $x \in B_\eps(\Omega) \setminus \Omega$. Then
\begin{equation}
\int_{\Omega \setminus B_\eps(\partial \Omega)} \deltaO(y)^{-a} |x-y|^{\lambda-N}dy
\le \eps^{\lambda-N-a}|\Omega| \le C \le \frac{C}{(1-a)\bigl(1+\deltaO(x)^{N-\lambda}\bigr)}, 
\label{first-long-lemma-second-est}
\end{equation}
so it remains to estimate the integral 
$$
\int_{\Omega \cap B_\eps(\partial \Omega)} \deltaO(y)^{-a} |x-y|^{\lambda-N}dy.
$$
Since $x \in B_\eps(\partial \Omega) \setminus \Omega$, we have
 $|x-y| \geq \deltaO(x)+\deltaO(y)$ and
	$$
|x-y|\geq |p(x)-y|-\deltaO(x)\geq |p(x)-p(y)|-\deltaO(x)-\deltaO(y)\qquad \text{for $y \in B_\eps(\partial \Omega) \cap \Omega$,} 
$$ 
which implies that
$$	
3|x-y|\geq |p(x)-p(y)|+\deltaO(x)+\deltaO(y)\qquad \text{for $y \in B_\eps(\partial \Omega) \cap \Omega$.}
$$
Consequently, for $x \in B_\eps(\Omega) \setminus \Omega$ we have
\begin{align}
\int_{B_\eps(\partial \Omega) \cap \Omega}& \deltaO(y)^{-a} |x-y|^{\lambda-N}dy \le 
3^{N-\lambda} \int_{B_\eps(\partial \Omega) \cap \Omega} \deltaO(y)^{-a}[|p(x)-p(y)|+\deltaO(y) + \deltaO(x)]^{\lambda-N}dy  \nonumber\\
&\le C \int_{\partial \Omega \times [0,\eps]}t^{-a}[t + \deltaO(x)+ |p(x)-z|]^{\lambda-N} T^*(dy)(z,t) \nonumber \\
&\le C \int_0^\eps t^{-a} 
\int_{\partial \Omega}[t + \deltaO(x)+ |p(x)-z|]^{\lambda-N} d\sigma(z)dt
\label{first-long-lemma-third-est-a}
\end{align}
Next we use the fact that, since $\partial \Omega$ is of class $C^2$, there exists $\mu=\mu(\Omega)>0$ and, for each $z \in \partial \Omega$, a $C^1$-Diffeomorphism 
$$
K_z: B_\mu(0) \subset \R^{N-1} \to B_\mu(z) \cap \partial \Omega
$$
with the property that 
$$
|K_z(v)-z|= |v| \qquad \text{for every $v \in B_\mu(0)$}
$$
and so that its partial derivatives remain bounded on $B_\mu(0)$ independently of $z$. Consequently, for every $t \in (0,\eps)$, 
\begin{align*}
&\int_{\partial \Omega}[t + \deltaO(x)+ |p(x)-z|]^{\lambda-N} d\sigma(z)
= \int_{\partial \Omega \cap B_\mu(p(x))}[t + \deltaO(x)+ |p(x)-z|]^{\lambda-N}d\sigma(z) + \int_{\partial \Omega \setminus B_\mu(p(x))}\dots\\
&\le C \int_{B_\mu(0)}[t + \deltaO(x)+ |v|]^{\lambda-N} dv + \mu^{\lambda-N}|\partial \Omega| \le C \int_{0}^\mu r^{N-2}[t + \deltaO(x)+r]^{\lambda-N}
dr + C\\
&= C(t+\deltaO(x))^{\lambda-1} \int_{0}^{\frac{\mu}{t+\deltaO(x)}}\tau^{N-2} \bigl(1 +\tau)^{\lambda-N}
d\tau+C\\
&\le  C (t+\deltaO(x))^{\lambda-1}\Bigl(1+  \int_{1}^{\infty} \tau^{\lambda-2}d\tau \Bigr)+ C \le  \frac{C}{1-\lambda}(t+\deltaO(x))^{\lambda-1}+C
\end{align*}
and therefore
\begin{align}
&\int_0^\eps t^{-a} 
\int_{\partial \Omega}[t + \deltaO(x)+ |p(x)-z|]^{\lambda-N} d\sigma(z)dt \label{first-long-lemma-third-est-b}
\\
&\le C \Bigl( \int_0^\eps t^{-a}\,dt + \frac{1}{1-\lambda}\int_0^\eps t^{-a}(t+\deltaO(x))^{\lambda-1}\,dt \Bigr) +C \le \frac{C}{1-a} + \frac{\kappa(a,\lambda,\deltaO(x))}{1-\lambda}\nonumber
\end{align}
by Lemma~\ref{basic-estimate}. Since 
$$
\kappa(a,\lambda,\deltaO(x))= 
\left \{
  \begin{aligned}
&\min \{(\lambda-a)^{-1}, (1-a)^{-1}+|\ln c|\}, &&\quad \text{if $\lambda>a$,}\\
&\deltaO^{\lambda-a}(x)\min\{ [(1-a)(a-\lambda)]^{-1}(1-\lambda),(1-a)^{-1} + |\ln \deltaO(x)|\}, &&\quad \text{if $\lambda<a$,}\\
&(1-a)^{-1}+|\ln \deltaO(x)|\bigr)\le C(1-a)^{-1}\bigl(1+|\ln \deltaO(x)|\bigr), && \quad \text{if $\lambda=a$,}
  \end{aligned}
\right.
$$
it follows that 
\begin{equation}
\int_0^\eps t^{-a} 
\int_{\partial \Omega}[t + \deltaO(x)+ |p(x)-z|]^{\lambda-N} d\sigma(z)dt  \le C \, m(a,\lambda,\deltaO(x)).
 \label{first-long-lemma-final-est}
\end{equation}
The claim now follows by combining (\ref{first-long-lemma-first-est}), (\ref{first-long-lemma-second-est}), (\ref{first-long-lemma-third-est-b}) and (\ref{first-long-lemma-final-est}) and by making $C$ larger if necessary. 
\end{proof}

\begin{cor}\label{est:l-new}
Let $N\geq 2$, $\Omega\subset \R^N$ be an open bounded set with $C^{2}$-boundary, $a \in(-1,1)$, $b \in [0,2]$, $\lambda \in [-N,1)$ and $\nu \in [\lambda-N,b]$. Then there is $C=C(N,\Omega)>0$ such that, for all $x\in \Omega$,
$$
	\int_{\R^N\backslash \Omega}\frac{|x-y|^{\nu}}{\deltaO(y)^{a}+\deltaO(y)^{N+b}}\ dy	\leq C \, m(a,\lambda,\deltaO(x)),
$$
where $m(\cdot,\cdot,\cdot)$ is defined in Lemma~\ref{int:la}.
\end{cor}

\begin{proof}
Let $x \in \Omega$. We split 
$$
	\int_{\R^N\backslash \Omega}\frac{|x-y|^{\nu}}{\deltaO(y)^{a}+\deltaO(y)^{N+b}}\ dy = 
	\int_{B_{r_\Omega}(x) \backslash \Omega}\dots dy + 
\int_{\R^N \setminus B_{r_\Omega}(x)}\dots dy
$$
with $r_\Omega:= \diam \:\Omega +1$ as before. In the following, the letter $C$ stands for positive constants depending only on $N$ and $\Omega$.
For $y \in \R^N \setminus B_{r_\Omega}(x)$ we have $\deltaO(y) \ge C |x-y|$ and thus
$$
\frac{|x-y|^{\nu}}{\deltaO(y)^{a}+\deltaO(y)^{N+b}}\le C |x-y|^{\nu-N-b}. 
$$
Therefore, since $r_\Omega \ge 1$ and $\nu \le b$, 
$$
\int_{\R^N \setminus B_{r_\Omega}(x)}\frac{|x-y|^{\nu}}{\deltaO(y)^{a}+\deltaO(y)^{N+b}}\, dy \le C \int_{r_\Omega}^\infty 
\tau^{\nu-1-b}d\tau = \frac{r_\Omega^{\nu-1-b}}{1+b-\nu} \le r_\Omega^{\nu-1-b}\le 1 \le C m(a,\lambda,\deltaO(x)).
$$
Moreover, by our assumptions on $\nu$, $\lambda$ and $b$, we have
$$
|x-y|^{\nu} \le r_\Omega^{\nu+ N-\lambda} |x-y|^{\lambda-N} \le r_\Omega^{2+ 2N}  |x-y|^{\lambda-N}
\qquad \text{for $y \in B_{r_\Omega}(x)$}
$$
and therefore, by Lemma~\ref{int:la}, 
$$ 
\int_{B_{r_\Omega}(x) \backslash \Omega} \frac{|x-y|^{\nu}}{\deltaO(y)^{a}+\deltaO(y)^{N+b}}\ dy
\le C 
\int_{B_{r_\Omega}(x) \backslash \Omega}\deltaO(y)^{-a}|x-y|^{\lambda-N} dy \le C m(a,\lambda,\deltaO(x)).
$$
\end{proof}

In Section~\ref{sec:thecases0} below, we will also need the following integral estimate.

\begin{lemma}
\label{k-s-eps-cor}
Let $N\geq 2$, and let $\Omega\subset \R^N$ be an open bounded set with $C^{2}$-boundary. For $s \in (0,1)$ and $\eps \in (0,\min \{s,1-s\})$ we have 
$$
K_{s,\eps}(\Omega):= \sup_{x \in \Omega} \int_\Omega \deltaO^{-s-\eps}(y) |x-y|^{2s-N}\ dy \le 
\frac{C}{(\min\{s,1-s\}-\eps)^3}.
$$
with a constant $C=C(N,\Omega)>0$.
\end{lemma}

\begin{proof}
For $s$, $\eps$ as above, we fix $\lambda= \frac{1}{2}(s+\eps+\min\{2s,1\}) \in (s+\eps, \min \{2s,1\})$. Moreover, we let $C$ denote constants depending only on $\Omega$ and $N$ in the following. Since $|x-y| \le \diam\: \Omega < \infty$ for $x,y \in \Omega$, we have, by Lemma \ref{int:la}, 
\begin{equation}
  \label{eq:K-s-eps-old-est}
  K_{s,\eps}(\Omega)\le [\diam\: \Omega]^{2s-\lambda} \sup_{x \in \Omega} \int_\Omega \deltaO^{-s-\eps}(y) |x-y|^{\lambda-N}\ dy \le r_\Omega^{2}
\sup_{x \in \Omega} \int_\Omega \deltaO^{-s-\eps}(y) |x-y|^{\lambda-N}\ dy
\end{equation}
with $r_\Omega$ given in Lemma~\ref{int:la}. For $x \in \Omega$, we write $r_x:= \frac{\deltaO(x)}{2}$ and 
  \begin{equation}
    \label{eq:lemma-2-5-integral-splitting}
  \int_\Omega \deltaO^{-s-\eps}(y) |x-y|^{\lambda-N}\ dy = \int_{B_{r_x}(x)} \dots \ dy + \int_{\Omega \setminus B_{r_x}(x)} \dots \ dy.
  \end{equation}
  Since $\deltaO(y) \ge r_x$ for $y \in B_{r_x}(x)$, we have
  \begin{equation}
    \label{eq:lemma-2-5-esasy-est}
  \int_{B_{r_x}(x)} \deltaO^{-s-\eps}(y) |x-y|^{\lambda-N}\ dy \le r_x^{-s-\eps} \int_{B_{r_x}(0)}|z|^{\lambda-N}\,dz = \frac{r_x^{\,\lambda-s-\eps}}{\lambda} \le \frac{r_\Omega}{s+\eps}.
  \end{equation}
  To estimate the second integral in (\ref{eq:lemma-2-5-integral-splitting}), we note that, since $\Omega$ is a bounded domain of class $C^2$, there exists $\rho \in (0,1)$ and a unique reflection map
  $$
  B_\rho(\partial \Omega) \to   B_\rho(\partial \Omega), \qquad x \mapsto \bar x
  := x -2(p(x)-x),
  $$
  at the boundary $\partial \Omega$. Here, as in the proof of Lemma~\ref{int:la}, $p(x)$ is the projection of $x \in B_\rho(\partial \Omega)$ onto $\partial \Omega$. If $x \in \Omega \setminus B_{\rho}(\partial \Omega)$, then we have $r_x \ge \frac{\rho}{2}$ and therefore 
  \begin{equation}
\label{eq:lemma-2-5-2}
\int_{\Omega \setminus B_{r_x}(x)} \deltaO^{-s-\eps}(y) |x-y|^{\lambda-N}\ dy
\le \Bigl(\frac{\rho}{2}\Bigr)^{\lambda-N} \int_{\Omega} \deltaO^{-s-\eps}(y)\ dy \le \frac{C}{1-s-\eps}
  \end{equation}
by the same estimate as in (\ref{eq:est-add-revision}). If $x \in B_{\rho}(\partial \Omega) \cap  \overline \Omega$, then we have 
$\bar x \in B_{\rho}(\partial \Omega) \setminus \Omega$ and
$$
\deltaO(x)= \deltaO(\bar x),\qquad |x-\bar x|= 2\deltaO(x)=4r_x.
$$
Consequently, 
$$
\Bigl(\frac{|\bar x-y|}{|x-y|}\Bigr)^{N-\lambda}\le \Bigl(1+ \frac{4 r_x}{|x-y|}\Bigr)^{N-\lambda} \le 5^{N-\lambda} \le 5^{N}\qquad \text{for $y \in \Omega \setminus B_{r_x}(x)$,}
$$
which implies that 
$$
\int_{\Omega \setminus B_{r_x}(x)} \deltaO^{-s-\eps}(y) |x-y|^{\lambda-N}\ dy \le
5^{N} \int_{\Omega} \deltaO^{-s-\eps}(y) |\bar x-y|^{\lambda-N}\ dy
$$
and therefore, by (\ref{eq:(i)-int-est}) applied to $\bar x$,
\begin{align}
  \int_{\Omega \setminus B_{r_x}(x)} \deltaO^{-s-\eps}(y) |x-y|^{\lambda-N}\ dy &\le C m(s+\eps,\lambda,\deltaO(\bar x))= C m(s+\eps,\lambda,\deltaO(x)) \nonumber\\
  &\le \frac{C}{(1-s-\eps)(1-\lambda)(\lambda-s-\eps)} \le \frac{C}{(\min\{s,1-s\}-\eps)^3}.\label{lemma-2-5-final-est} 
\end{align}
The claim follows by combining (\ref{eq:K-s-eps-old-est})--(\ref{lemma-2-5-final-est}). 
\end{proof}

\section{The solution map and related operators}
\label{sec:solution-map-related}

Throughout this section, we consider the case $N \ge 2$. We introduce first some important notation related to a fractional Poisson problem in an open bounded set $\Omega \subset \R^N$. Recall that, for $s \in (0,1)$, the fundamental solution of $(-\Delta)^s$ is given by $F_s(x,y)=F_s(|x-y|)$, where
\[
F_s(z)=\kappa_{N,s}|z|^{2s-N},\quad\kappa_{N,s}=\frac{\Gamma(\frac{N}{2}-s)}{4^s\pi^{N/2}\Gamma(s)}=\frac{s \Gamma(\frac{N}{2}-s)}{4^s\pi^{N/2}\Gamma(1+s)}.
\]
We also note that 
$$
\frac{\kappa_{N,s}}{s} = \frac{\Gamma(\frac{N}{2}-s)}{4^s\pi^{N/2}\Gamma(1+s)}
\to c_N:= \frac{\Gamma(\frac{N}{2})}{\pi^{N/2}}\qquad \text{as $s \to 0^+$.}
$$
The convolution with the fundamental solution is usually called the {\em Riesz operator} 
$$
f \mapsto \cFs_s f:= F_s * f.
$$
By the Hardy-Littlewood-Sobolev inequality (see \cite{LL01,L83}) this convolution defines a continuous linear map $\cFs_s: L^r(\R^N) \to L^{p(r,s)}(\R^N)$ for 
$$
r \in (1,\frac{N}{2s}) \qquad \text{and}\qquad p(r,s):= \frac{rN}{N-2sr}.
$$
We note the following semigroup property of the operator family $\cFs_s$, $s \in (0,1)$. 

\begin{lemma}
\label{semigroup}
Let $s,\sigma>0$ be given with $\sigma +s <1$. Moreover, let 
$r \in (1,\frac{N}{2(s+\sigma)})$. Then 
\begin{equation}
  \label{eq:semigroup-exponents}
r \in (1,\frac{N}{2\sigma}), \quad p(r,\sigma) \in (1,\frac{N}{2s}) 
\qquad \text{and} \qquad p(p(r,\sigma),s)= p(r,\sigma+s).
\end{equation}
Moreover, for $g \in L^{r}(\R^N)$ we have 
\begin{equation}
  \label{eq:semigroup-function}
\cFs_{\sigma+s} g = \cFs_{s} (\cFs_\sigma g) \quad \text{in $L^{p(r,\sigma+s)}(\R^N)$.}
\end{equation}
\end{lemma}

\begin{proof}
Direct computation yields (\ref{eq:semigroup-exponents}). Moreover, (\ref{eq:semigroup-function}) is true, by a Fourier transform argument, for functions $g \in C^\infty_c(\R^N)$, and it follows by density for functions $g \in L^r(\R^N)$ (see e.g. \cite{L72,G08}). 
\end{proof}

Throughout the remainder of this section, let $\Omega \subset \R^N$ be a fixed open and bounded set with $C^{2}$-boundary. We recall the following result from \cite{nicola} (see also \cite{BB99,BKK08,BJK19}).

\begin{lemma}
\label{poisson-representation-new-zero}
Let $g: \R^N \setminus \Omega  \to \R$ be a measurable function which is bounded in a relative neighborhood of $\partial \Omega$ in $\R^N \setminus \Omega$ and which satisfies 
$$
\int_{\R^N} \frac{|g(x)|}{(1+|x|)^{N+2s}}\,dx < \infty.
$$
Then there exists a unique function $u: \R^N \to \R$ which is $s$-harmonic in $\Omega$, bounded in a neighborhood of $\partial \Omega$ and satisfying  
$u \equiv g$ in $\R^N \setminus \Omega$. 
\end{lemma}

Using this result, we may define 
$$
\R^{2N}_* := \R^N \times \R^N \setminus \{(x,x)\:: \: x \in \R^N\}
$$
and define a function $H_s: \R^{2N}_* \to \R$, associated with $\Omega$, as follows: For fixed $x\in \Omega$, $H_s(x,\cdot): \R^N \to \R$ denotes the unique solution in $L^\infty(\R^N)$ of
\[
\left\{\begin{aligned}
(-\Delta)^s_yH_s(x,\cdot)&=0&&\text{in $\Omega$,}\\
H_s(x,\cdot)&=F_s(x-\,\cdot)&& \text{on $\R^N\setminus \Omega$.}
\end{aligned}\right.
\]
Moreover, if $x \in \R^N \setminus \Omega$, we set 
$H_s(x,\cdot): = F_s(x-\cdot)$ on $\R^N \setminus \{x\}$.
By the maximum principle, we then have 
\begin{equation}
  \label{eq:H-F-est}
0 \le H_s(x,y) \le F(x-y) \qquad \text{for $x,y \in \R^{2N}_*$.}  
\end{equation}
Consequently, for $r \in (1,\frac{N}{2s})$, we can define an operator 
$$
\cHs_s: L^{r}(\Omega) \to L^{p(r,s)}(\R^N), \qquad [\cHs_s f](x) = \int_{\Omega}H_s(x,y)f(y)\,dy \qquad \text{for $x \in \R^N$.}
$$

The following properties are well known.
 
\begin{lemma}[see \cite{BB99,S07,G15:2,RS16}]
\label{F-H-properties}
Let $r \in (1,\frac{N}{2s})$. 
\begin{enumerate}
\item[(i)] Let $g \in L^{r}(\R^N)$ and $u: = \cFs_s g \in L^{p(r,s)}(\R^N)$. 
If $g \in C^\alpha_{loc}(\Omega)$ for some $\alpha \in (0,1)$, then $u \in C^{2s+\beta}_{loc}(\Omega)$ for $\beta \in (0,\alpha]$ with $2s + \beta \not \in \N$, and 
$$
(-\Delta)^s u = g \qquad \text{in $\Omega$.}
$$
\item[(ii)] Let $g \in L^{r}(\Omega) \cap C^\alpha_{loc}(\Omega)$ for some $\alpha \in (0,1)$ and $u: = \cHs_s g \in L^{p(r,s)}(\R^N)$. Then $u \in C^{2s+\beta}_{loc}(\Omega)$ for $\beta \in (0,\alpha]$ with $2s + \beta \not \in \N$, and 
$$
(-\Delta)^s u = 0 \qquad \text{in $\Omega$.}
$$
\end{enumerate}
\end{lemma}

We note that, under the assumptions of Lemma~\ref{F-H-properties}, the function $u = \cHs_s g$ solves the problem 
\begin{equation}
  \label{eq:label-Hs-problem}
(-\Delta)^s u = 0 \qquad \text{in $\Omega$}, \qquad u \equiv \cFs_s \cEO g \qquad \text{in $\R^N \setminus \Omega$}
\end{equation}
and hence it even holds that $u\in C^{\infty}(\Omega)$ (see \cite{BB99}). Here and in the following, for a real-valued function $g$ defined on $\Omega$, we let $\cEO  g$ denote the trivial extension of $g$ to $\R^N$, i.e., $\cEO  g\equiv g$ in $\Omega$ and $\cEO g  \equiv 0$ in $\R^N \setminus \Omega$. Note that, in general, $\cEO g$ is a discontinuous function.

Moreover, if $g$ is a real-valued function defined on $\R^N$, we let $\cRO g$ denote the restriction of $g$ to $\Omega$. We note the trivial identity
\begin{equation}
  \label{eq:trivial-identity}
\cEO \cRO g = g 1_\Omega \qquad \text{for any function $g: \R^N \to \R$.} 
\end{equation}

We now consider the associated Green function $G_s: \R^{2N}_* \to \R^N$ defined by 
\begin{equation}
  \label{eq:splitting-prelim-1}
G_s(x,y)=F_s(x-y)-H_s(x,y) \qquad \text{for $x,y \in \R^{2N}_*$,}
\end{equation}
By (\ref{eq:H-F-est}), we have 
\begin{equation}
  \label{eq:G-F-est}
0 \le G_s(x,y) \le F(x-y) \qquad \text{for $x,y \in \R^{2N}_*$,}  
\end{equation}
so, for $r \in (1,\frac{N}{2s})$ we can define an operator 
\begin{equation}
  \label{eq:G-op-definition}
  \begin{aligned}
&\cGs_s: = \cFs_s \cEO - \cHs_s : \quad L^{r}(\Omega) \to L^{p(r,s)}(\R^N),\\
&[\cGs_s f](x) = \int_{\Omega}G_s(x,y)f(y)\,dy \qquad \text{for $x \in \R^N$.}
  \end{aligned}
\end{equation}

 It can be seen that the functions $H_s$ and $G_s$ are symmetric with respect to reflection of coordinates $x$ and $y$ (see e.g. \cite{K97}). We also have the following. 

\begin{lemma}
\label{G-properties}
Let $g \in L^{r}(\Omega)$ for some $r \in (1,\frac{N}{2s})$. Then $u:= \cGs_s g \in L^{p(r,s)}(\R^N)$ and the following holds.
\begin{enumerate}
\item[(i)] If $g \in C^\alpha_{loc}(\Omega)$ for some $\alpha \in (0,1)$, then $u:= \cGs_s g \in L^{p(r,s)}(\R^N) \cap C^{2s+\beta}_{loc}(\Omega)$ for $\beta \in (0,\alpha]$ with $2s + \beta \not \in \N$, and 
  \begin{equation}
    \label{eq:green-problem}
\left\{
\begin{aligned}
(-\Delta)^s u &= g &&\qquad \text{in $\Omega$,}\\
u &= 0 &&\qquad \text{in $\R^N \setminus \Omega$.}  
\end{aligned}
\right.
  \end{equation}
\item[(ii)] If $g \in L^\infty(\Omega)$, then $u \in C^{s}(\R^N) \cap C^{2s-\eps}_{loc}(\Omega)$ for every $\eps \in (0,2s)$, and  
there are constants $C_1=C_1(\Omega)>0$ and $C_2=C_2(\Omega)>0$ such that
\begin{equation}\label{omega-bar-reg}
\|u\|_{C^s(\overline \Omega)} \le  C_1 \|g\|_{L^{\infty}(\Omega)}
\end{equation}
and 
\begin{equation}\label{sec:notation:boundary-reg}
|u(x)|\leq C_2\|g\|_{L^{\infty}(\Omega)}\deltaO(x)^s\quad \text{ for all $x\in \Omega$.}
\end{equation}
\end{enumerate}
\end{lemma}

\begin{proof}
(i) immediately follows from Lemma~\ref{F-H-properties} and the definition of the operator $\cGs_s$.\\
(ii) The interior regularity property has been proven in \cite{S07,G15:2,RS16,KM17}. The estimates (\ref{omega-bar-reg}) and (\ref{sec:notation:boundary-reg}) have been proven in \cite{RS12}. Moreover, the independence of $C_1$ and $C_2$ of $s$ can be seen by following closely the proof given in \cite{RS12}. For the reader's convenience we include a proof of this fact in Appendix \ref{appendix}.
\end{proof}

\begin{remark}
\label{green-weak-solution}
If $g \in L^2(\Omega)$, then $u:= \cGs_s g$ is also characterized as the unique weak solution $u \in \cH^s_0(\Omega)$ of (\ref{eq:green-problem}) as defined in the introduction. Indeed, as already noted in the introduction, it follows from the embedding $\cH^s_0(\Omega) \hookrightarrow L^2(\R^N)$ and Riesz' representation theorem that for every $g \in L^2(\Omega)$, there is a unique function 
$\tilde u \in \cH^s_0(\Omega)$ satisfying $(-\Delta)^s \tilde u=g$ in $\Omega$ in {\em weak sense}, i.e. 
$$
\cE_s(\tilde u ,v) = \int_{\Omega} fv\,dx \qquad \text{for every $v \in \cH^s_0(\Omega)$.}
$$
In the case where $g \in C^\infty_c(\Omega)$, it is easy to see that $u:= \cGs_s g \in \cH^s_0(\Omega)$ is a weak solution and therefore $u =  \tilde u$.
In the general case $g \in L^2(\Omega)$, we can argue by approximation with a sequence of functions $g_n \in C^\infty_c(\Omega)$ satisfying $g_n \to g$ in $L^2(\Omega)$ and hence also $g_n \to g$ in $L^{\frac{2N}{N+2s}}(\Omega)$. We can then use the fact that the corresponding weak solutions $\tilde u_n$ satisfy $\tilde u_n \to \tilde u$ in $\cH^s_0(\Omega)$ together with continuity of the Green operator $\cGs_s : L^{\frac{2N}{N+2s}}(\Omega) \to L^{\frac{2N}{N-2s}}(\R^N)$ and the Sobolev embedding $\cH^s_0(\Omega) \hookrightarrow L^{\frac{2N}{N-2s}}(\R^N)$ to conclude that $u= \tilde u$.
\end{remark}

Next we recall a general Poisson kernel representation formula for harmonic functions in $\Omega$ with prescribed values in $\R^N \setminus \Omega$, see e.g. \cite[Theorem 1.2.3 and Lemma 3.1.2)]{nicola} (see also \cite{B99,BKK08}). 

\begin{lemma}
\label{poisson-representation-new}
Under the assumptions of Lemma~\ref{poisson-representation-new-zero}, we have 
$$
u(x)= -\int_{\R^N\setminus \Omega}g(z) (-\Delta)^s_z G_s(x,z)\ dz \qquad \text{for $x \in \Omega$.}
$$
\end{lemma}

An immediate consequence of Lemma~\ref{poisson-representation-new} are the following formulas, which are stated in \cite[Lemma 3.2.vi)]{nicola}: 
\begin{align*}
F_s(x- \,\cdot\,) &= -\int_{\R^N\setminus \Omega}F_s(x-z) (-\Delta)^s_z G_s(\cdot ,z)\ dz \qquad \text{in $\Omega$ for $x \in \R^N \setminus \overline{\Omega}$,}\\
H_s(x,\cdot) &= -\int_{\R^N\setminus \Omega}F_s(x-z) (-\Delta)^s_z G_s(\cdot,z)\ dz \qquad \text{in $\Omega$ for $x \in \Omega$.}
\end{align*}
Using the facts that $H_s(x,y)= F_s(x-y)$ if $x \in \R^N \setminus \Omega$ or $y \in \R^N \setminus \Omega$, we can write these identities in a compact form as \begin{equation*}
H_s(x,y)=-\int_{\R^N\setminus \Omega}F_s(x-z)(-\Delta)^s_z G_s(y,z)\ dz \qquad \text{for $x \in \R^N \setminus \partial \Omega$, $y \in \Omega$.}
\end{equation*}
Since, for fixed $y \in \Omega$, both sides of the equation are continuous in $x \in \R^N \setminus \{y\}$, the latter identity also holds for $x \in \partial \Omega$, i.e., 
\begin{equation}
  \label{eq:splitting-prelim-2}
H_s(x,y)=-\int_{\R^N\setminus \Omega}F_s(x-z)(-\Delta)^s_z G_s(y,z)\ dz \qquad \text{for $x \in \R^N$, $y \in \Omega$.}
\end{equation}

In order to derive a useful decomposition of the operator $\cGs_s$, we need an estimate for $(-\Delta)^{s}[\cGs_s f]$ on $\R^N \setminus \overline{\Omega}$ for $f \in L^\infty(\Omega)$. We have the following. 
\begin{lemma}\label{qs-bound}
Let $f\in L^{\infty}(\Omega)$ and $0< s< 1$.  Then 
\begin{equation}\label{qs-bound2}
|(-\Delta)^{\sigma}\cGs_s f(x)|\leq \sigma C \|f\|_{L^{\infty}(\Omega)} \frac{1 + \deltaO(x)^{s-2\sigma} |\ln \deltaO(x)|}{1+\deltaO(x)^{N+2\sigma}} \quad\text{ for $x\in \R^N\setminus \overline{\Omega}$, $\sigma\in(s/2,s]$}
\end{equation}
with a constant $C=C(N,\Omega)>0$.
In particular, defining 
$$
\cQs_s f: \R^N \to \R, \qquad [\cQs_s f](x)= \left \{
  \begin{aligned}
  &-(-\Delta)^s [\cGs_s f](x),&&\qquad x \in \R^N \setminus \overline{\Omega},\\
  &0,&& \qquad x \in \overline{\Omega},  
  \end{aligned}
\right.
$$
we have that
\begin{equation}\label{Q:est}
|\cQs_s f(x)|\leq  s \,C \|f\|_{L^{\infty}(\Omega)} \frac{1 + \deltaO(x)^{-s} |\ln \deltaO(x)|}{1+\deltaO(x)^{N+2s}}\quad\text{for $x\in \R^N \setminus \Omega$, $s\in(0,1)$.}
\end{equation}
Consequently, $\cQs_s$ is a continuous linear map $L^\infty(\Omega) \to L^r(\R^N)$ for $r \in [1,\frac{1}{s})$. In addition, we have the estimate
\begin{equation}\label{Q:est-1}
|\cQs_s f(x)|\leq  \frac{C \|f\|_{L^{\infty}(\Omega)}}{\deltaO(x)^{s}+\deltaO(x)^{N+2s}}\quad\text{for $x\in \R^N \setminus \Omega$, $s\in(0,1)$}
\end{equation}
with a constant $C=C(N,\Omega)>0$.
\end{lemma}

\begin{proof}
In the following, the letter $C$ denotes positive constants depending only on $N$ and $\Omega$. By \eqref{sec:notation:boundary-reg}
\[
|\cGs_s f(x)|\leq C\deltaO(x)^{s}\|f\|_{L^{\infty}(\Omega)}\quad\text{ for $x\in \Omega$,}
\]
Hence, for $x\in \R^N\setminus\overline{\Omega}$ and $\sigma\in(s/2,s]$,
\begin{equation}\label{eq:qs-bound:11}
\Big|(-\Delta)^\sigma \cGs_s f(x)\Big|\leq c_{N,\sigma} C \|f\|_{L^{\infty}(\Omega)}\int_{\Omega} \frac{\deltaO(y)^{s}}{|x-y|^{N+2\sigma}}\ dy.
\end{equation}
where $0<c_{N,\sigma}=\frac{\sigma 4^\sigma\Gamma(\frac{N}{2}+\sigma)}{\pi^{\frac{N}{2}}\Gamma(1-\sigma)}\leq \sigma C$ and 
\begin{align*}
\int_{\Omega} &\frac{\deltaO(y)^{s}}{|x-y|^{N+2\sigma}}\ dy \le \frac{C m(-s,-2\sigma,\deltaO(x))}{1+\deltaO(x)^{N+2\sigma}}\\
&\le \frac{C}{(1+s)(1+\deltaO(x)^{N+2\sigma})} \Bigl(1+ \frac{\deltaO(x)^{-s}(1+|\ln \deltaO(x)|)}{1+2\sigma}\Bigr) 
\le \frac{C(1+\deltaO(x)^{-s}|\ln \deltaO(x)|)}{1+\deltaO(x)^{N+2\sigma}}
\end{align*}
by Lemma \ref{int:la} (with $\lambda=-2\sigma$ and $a=-s$). Thus \eqref{qs-bound2} follows, and \eqref{Q:est} is a direct consequence of \eqref{qs-bound2}. Using also that 
$$
\frac{m(-s,-2s,\deltaO(x))}{1+\deltaO(x)^{N+2\sigma}}\le \frac{C\bigl(1+ \frac{\deltaO(x)^{-s}}{s}\bigr)}{(1+s)(1+\deltaO(x)^{N+2\sigma})}  
\le \frac{C}{s\bigl(\deltaO(x)^s+\deltaO(x)^{N+2\sigma}\bigr)},
$$
we obtain the alternative estimate
$$
|\cQs_s f(x)| \leq  \frac{c_{N,s} C \|f\|_{L^{\infty}(\Omega)}}{s\bigl(\deltaO(x)^{s}+\deltaO(x)^{N+2s}\bigr)}
\le \frac{C \|f\|_{L^{\infty}(\Omega)}}{\deltaO(x)^{s}+\deltaO(x)^{N+2s}}
$$
for $x \in \R^N \setminus \Omega$ and $s \in (0,1)$, as claimed in (\ref{Q:est-1}).
\end{proof}

\begin{remark}
\label{Q-positivity-preserving}
We also note that the operator $\cQs_s$ defined in Lemma~\ref{qs-bound} is positivity preserving. Indeed, if $f \in L^\infty(\Omega)$ is nonnegative, then $\cGs_s f$ is nonnegative in $\Omega$ and vanishes on $\R^N \setminus \Omega$. Consequently, 
$$
[\cQs_s f](x)= c_{N,s} \int_{\Omega}\frac{[\cGs_s f](y)}{|x-y|^{N+2s}}dy \ge 0 \qquad \text{for $x \in \R^N \setminus \Omega$.}
$$
By definition, the operator $\cFs_s$ is also positivity preserving, and so are the operators $\cGs_s$ and $\cHs_s$ by (\ref{eq:H-F-est}) and (\ref{eq:G-F-est}).
Since all of these operators are linear, it follows that they are order preserving and therefore give rise to pointwise inequalities of the form 
$$
|\cQs_s f| \le \cQs_s |f|, \qquad |\cFs_s f| \le \cFs_s |f|,\qquad |\cHs_s f| \le \cHs_s |f| \qquad \text{and}\qquad |\cGs_s f| \le \cGs_s |f| 
$$
for functions $f$ in the respective domain of definition. 
\end{remark}

By (\ref{eq:splitting-prelim-2}), we can write, for $f \in L^\infty(\Omega)$ and $x \in \R^N$,
\begin{align*}
[\cHs_s f](x)&=- \int_{\Omega} \int_{\R^N\setminus \Omega}F_s(x-z) (-\Delta)^s_z G_s(y,z)\ dz f(y)\,dy\\
&=-\int_{\R^N\setminus \Omega} F_s(x-z) (-\Delta)^s_z \int_{\Omega}  G_s(y,z) f(y)\,dy\ dz= \int_{\R^N} F_s(x-z) [\cQs_s f](z) dz = [\cFs_s \cQs_s f](x)
\end{align*}
with the operator $\cQs_s$ defined in Lemma~\ref{qs-bound}.  Here the application of Fubini's theorem is justified by the estimate in 
Lemma~\ref{qs-bound}. In short, 
\begin{equation}
  \label{eq:Hs-identity}
\cHs_s f = \cFs_s \cQs_s f \qquad \text{in $\R^N$ for $f \in L^\infty(\Omega)$.}
\end{equation}
This yields the decomposition
\begin{equation}
\label{Q:rep}
\cGs_s f= \cFs_s \cEO f -\cHs_s f =  \cFs_s \Bigl(\cEO f- \cQs_s f\Bigr) 
 \qquad \text{in $\R^N$ for $f \in L^\infty(\Omega)$,}
\end{equation}
which turns out to be highly useful for our purposes.  It is natural to also define 
$$
G_0 f = \cEO f \qquad \text{for $f \in L^\infty(\Omega)$.}
$$
Since $(-\Delta)^0$ is the identity operator, we thus have that $\cQs_0 f = 0$ for $f \in L^\infty(\Omega)$. It is also consistent to identify $F_0$ with the identity operator on $L^r(\R^N)$ for $r>1$. With these definitions, we can now analyze continuity and differentiability of the solution map 
$$
s\mapsto u_s:= \cGs_s f,\qquad s\geq 0
$$
for $f \in L^\infty(\R^N)$. 

\section{Differentiability of the solution map at \texorpdfstring{$s=0$}{s=0}}\label{sec:thecases0}

The goal of the present section is to derive the following bound on difference quotients, which give the proof of the differentiability of the solution map at $s=0$. Throughout this section, we assume that $\Omega\subset\R^N$, $N \ge 2$ is an open bounded set with $C^{2}$ boundary. The following is the main result of this section.

\begin{prop}\label{derivative-zero}
Let $f\in C^{\alpha}(\overline{\Omega})$ for some $\alpha>0$. Then 
\begin{equation}
\label{derivative-zero-est-1}
\limsup_{s \to 0^+}\frac{1}{s} \Big\|\deltaO^{\eps}\Bigl( \frac{G_sf-f}{s}+\loglap \cEO f\Bigr) \Big\|_{L^\infty(\Omega)} \le C \|f\|_{C^\alpha(\overline \Omega)}\qquad \text{for every $\epsilon>0$}
\end{equation}                                                     
with a constant $C= C(N,\Omega,\eps,\alpha)>0$.  In particular,
\begin{align*}
\text{$\cGs_s f \to f$\quad and \quad $\frac{\cGs_sf-f}{s} \to - \loglap \cEO f$ \quad almost uniformly in $\Omega$ as $s \to 0^+$.}
\end{align*}
\end{prop}

To show Proposition \ref{derivative-zero}, we fix $f\in C^{\alpha}(\overline{\Omega})$, and we first note that  
\begin{equation}\label{sec:thecases0-separation}
\frac{\cGs_sf-\cEO f }{s}+1_{\Omega} \loglap \cEO f
=1_\Omega \Big(\frac{\cFs_s-\textnormal{id}}{s}\cEO f+\loglap \cEO f\Big)  - 
1_\Omega \frac{\cHs_sf}{s}.
\end{equation}
The following lemmas provide bounds on the terms in \eqref{sec:thecases0-separation}. We recall here that 
\[
\frac{F_s(z)}{s}=\frac{\Gamma(\frac{N}{2}-s)}{s4^s\pi^{N/2}\Gamma(s)}|z|^{2s-N}=\frac{\Gamma(\frac{N}{2}-s)}{4^s\pi^{N/2}\Gamma(s+1)}|z|^{2s-N}\to c_N|z|^{-N}\quad\text{as $s\to 0^+$ for $z \in \R^N \setminus \{0\}$.}
\]

\begin{lemma}\label{rn-case}\hspace{1em}
  \begin{enumerate}
  \item[(i)] If $f\in C^{\alpha}_c(\R^N)$ for some $\alpha>0$, then  
    \begin{equation}
      \label{eq:rn-case-first-claim}
\Big|\Bigl[\frac{\cFs_s-\textnormal{id}}{s}f\Bigr] (x)+[\loglap f](x)\Big|\leq s\ C \|f\|_{C^{\alpha}(\R^N)}\quad\text{ for $x\in \R^N$, $s \in (0,\frac{1}{2}]$}
    \end{equation}
with some constant $C= C(N,\alpha,\supp f)>0$.
\item[(ii)] If $f\in C^{\alpha}(\overline{\Omega})$ for some $\alpha>0$, then 
  \begin{equation}
    \label{eq:rn-case-second-claim}
\Big|\Bigl[\frac{\cFs_s-\textnormal{id}}{s}\cEO f\Bigr] (x)+[\loglap \cEO f](x)\Big|\leq \frac{s\ C}{\eps} \deltaO(x)^{-\epsilon}\|f\|_{C^{\alpha}(\overline{\Omega})}
\end{equation}
for $x\in \Omega$, $\epsilon \in (0,1)$, $s\in(0,\frac{1}{2}]$ with some constant $C=C(N,\Omega,\alpha)>0$.
  \end{enumerate}
\end{lemma}

\begin{proof}
(i) Let $f\in C^{\alpha}_c(\R^N)$ for some $\alpha>0$ and $R>0$ with $\supp f \subset B_R(0)$. In the following $C>0$ denotes possibly different constants depending only on $N,R$ and $\alpha$. Recall that $c_N=\frac{2}{|\partial B_1|}$, and thus
\[
\int_{B_1}|y|^{2s-N}\ dy=\frac{1}{s\, c_N}.
\]
For $x\in \R^N$, we thus have 
$$
\frac{\cFs_s}{s}f(x)=\frac{\kappa_{N,s}}{s}\int_{B_1}(f(x+y)-f(x))|y|^{2s-N}\ dy+ \frac{\kappa_{N,s}}{s^2c_N}f(x)+\frac{\kappa_{N,s}}{s}\int_{\R^N\setminus B_1}f(x+y)|y|^{2s-N}\ dy
$$
and therefore
\begin{align}
&|\frac{\cFs_s-\textnormal{id}}{s}f(x)+\loglap f(x)| \label{rn-case-first-est}\\
&\leq \int_{B_1}\frac{|f(x+y)-f(x)|}{|y|^N}\Big|\frac{\kappa_{N,s}}{s}|y|^{2s}-c_N\Big|\ dy+ \int_{\R^N\setminus B_1}\frac{|f(x+y)|}{|y|^N}\Big|\frac{\kappa_{N,s}}{s}|y|^{2s}-c_N\Big|\ dy\nonumber\\
&\qquad\qquad +|f(x)|\Big|\frac{\kappa_{N,s}}{s^2c_N}-\frac{1}{s}+\rho_N\Big|.\nonumber
\end{align}
Next we recall that, since $N \ge 2$, the function $s \mapsto \tau(s):= \frac{\kappa_{N,s}}{s}= \frac{\Gamma(\frac{N}{2}-s)}{4^s\pi^{N/2}\Gamma(1+s)}$
admits a smooth extension on $(-1,1)$ with $\tau(0)= c_N$ and $\tau'(0)=-c_N \rho_N$. Hence, this implies that  
$\big|\frac{\kappa_{N,s}}{s^2c_N}-\frac{1}{s}+\rho_N\big| \le s C$ for $s \in (0,\frac{1}{2})$ and 
$$
\Big|\frac{\kappa_{N,s}}{s}|y|^{2s}-c_N\Big|  \le 
\Big|\frac{\kappa_{N,s}}{s}-c_N\Big||y|^{2s} + c_N \Bigl| |y|^{2s} -1 \Big|
\leq s\, C(1+\big|\ln|y|\big|+
\big|\ln|y|\big| |y|^{2s}  )
$$
for $s \in (0,\frac{1}{2}]$ and $y \in \R^N.$ Consequently,
\begin{align}
\int_{\R^N\setminus B_1}&\frac{|f(x+y)|}{|y|^N}\Big|\frac{\kappa_{N,s}}{s}|y|^{2s}-c_N\Big|\ dy \le  s\, C\|f\|_{L^{\infty}(\R^N)} \int_{B_R(-x) \setminus B_1}
 \frac{1+\big|\ln|y|\big|+
\big|\ln|y|\big| |y|^{2s}}{|y|^N}\,dy \nonumber\\ 
&\leq s\, C\|f\|_{L^{\infty}(\R^N)} |B_R(0)| \le s\, C\|f\|_{L^{\infty}(\R^N)} 
\label{rn-case-second-est}
\end{align}
and
\begin{align}
\int_{B_1}&\frac{|f(x+y)-f(x)|}{|y|^N}\Big|\frac{\kappa_{N,s}}{s}|y|^{2s}-c_N\Big|\ dy\notag \nonumber\\
&\leq Cs \|f\|_{C^{\alpha}(\R^N)}
\int_{B_1} |y|^{\alpha-N}(1+\big|\ln|y|\big|+
\big|\ln|y|\big| |y|^{2s})\,dy \le sC \|f\|_{C^{\alpha}(\R^N)}.\label{rn-case-third-est}
\end{align}
Combining (\ref{rn-case-first-est}), (\ref{rn-case-second-est}) and (\ref{rn-case-third-est}), we conclude that (\ref{eq:rn-case-first-claim}) holds.\\
(ii) Let $R>0$ be chosen with $\Omega \subset B_R(0)$. Then the estimates (\ref{rn-case-first-est}) and (\ref{rn-case-second-est}) still hold with $f$ replaced by $\cEO f$. Moreover, for $x \in \Omega$ and $\eps \in (0,1)$ we have, with $\delta_x:= \deltaO(x)$,
\begin{align*}
\int_{B_1}&\frac{|\cEO f(x+y)-\cEO f(x)|}{|y|^N}\Big|\frac{\kappa_{N,s}}{s}|y|^{2s}-c_N\Big|\ dy\\
\leq& sC  \|f\|_{C^{\alpha}(B_{\delta_x}(x))}\int_{B_{\delta_x}}
|y|^{\alpha-N}(1+\big|\ln|y|\big|+
\big|\ln|y|\big| |y|^{2s}) dy\\
&+ sC \|\cEO f\|_{L^{\infty}(\R^N)} 
\int_{B_1\setminus B_{\delta_x}}|y|^{-N}(1+\big|\ln|y|\big|+
\big|\ln|y|\big| |y|^{2s})\ dy\\
\leq& sC \|f\|_{C^{\alpha}(\overline \Omega)} 
+sC \|\cEO f\|_{L^{\infty}(\R^N)} 
\int_{B_1\setminus B_{\delta_x}}|y|^{-N-\eps}\ dy\\
\leq& sC \|f\|_{C^{\alpha}(\overline \Omega)} 
 + \frac{sC}{\eps} \|f\|_{L^{\infty}(\Omega)}\delta_x^{-\epsilon} \le \frac{sC}{\eps} \|f\|_{C^{\alpha}(\overline \Omega)}\deltaO(x)^{-\epsilon}. 
\end{align*}
This gives (\ref{eq:rn-case-second-claim}). 
\end{proof}

\begin{remark}
\label{log-bound-remark-1}
Let $f\in C^{\alpha}(\overline{\Omega})$ for some $\alpha>0$. Similarly as in the proof of Lemma~\ref{rn-case}(ii), we may prove the bound 
  \begin{equation}
    \label{eq:log-bound-1}
\big|[\loglap \cEO f](x)\big| \leq C \|f\|_{C^{\alpha}(\overline{\Omega})}(1+ |\ln \deltaO(x)|)\qquad \text{for $x\in \Omega$}
\end{equation}
with some constant $C=C(N,\Omega,\alpha)>0$.  Moreover, arguing as in \cite[Proposition 2.2 (ii)]{CW18}, we have that $\loglap \cEO f\in C^{\alpha-\epsilon}_{loc}(\Omega)$ for every $\epsilon>0$.
\end{remark}

\begin{lemma}\label{separate-case}
Let $f\in L^\infty(\Omega)$. Then 
\[
\limsup_{s \to 0^+} \frac{1}{s^2}\|\deltaO^\eps \cHs_s f\|_{L^\infty(\Omega)} \le C \|f\|_{L^\infty(\Omega)}\qquad \text{ for every $\eps>0$}
\]
with a constant $C=C(N,\Omega,\eps)$.
\end{lemma}

\begin{proof}
It suffices to consider $\eps \in (0,\frac{1}{4})$.  Let $s \in (0,\frac{\eps}{2})$. In the following, $C>0$ denotes constants depending only on $N,\Omega$ and $\eps$. Let $x \in \Omega$. Since 
$$
\frac{1}{s}[\cHs_s f](x) = \frac{1}{s} [\cFs_s \cQs_s f](x)= \frac{\kappa_{N,s}}{s}
\int_{\R^N\setminus \Omega}|x-y|^{2s-N} \cQs_s f(y)\,dy,
$$
we have, by Lemma~\ref{qs-bound},
\begin{align*}
\bigl|\frac{1}{s}&[\cHs_s f](x)\bigr| \le \kappa_{N,s}\, C \|f\|_{L^\infty(\Omega)} \int_{\R^N\setminus \Omega}\frac{(1+\deltaO(y)^{-s}|\ln \deltaO(y)|)|x-y|^{2s-N}}{1+\deltaO(y)^{N+2s}}\,dy\\
&\le s\,C \|f\|_{L^\infty(\Omega)} \int_{\R^N\setminus \Omega}\frac{(1+\deltaO(y)^{-s-\frac{\eps}{2}})|x-y|^{2s-N}}{1+\deltaO(y)^{N+2s}}\,dy.
\end{align*}
Therefore, by Corollary~\ref{est:l-new}, applied with $a=0$ and $a=s+\frac{\eps}{2}$, $b= 2s$ and $b=3s+\frac{\eps}{2}$, $\lambda=0$, $\nu = 2s-N$,
\begin{align*}
\bigl|\frac{1}{s}&[\cHs_s f](x)\bigr| \le
s\,C \|f\|_{L^\infty(\Omega)}  \Bigl(m(0,0,\deltaO(x)) + m(s+\frac{\eps}{2},0,\deltaO(x))\Bigr)\\
& \le s \, C  \|f\|_{L^\infty(\Omega)} \Bigl(1+|\ln \deltaO(x)| + \deltaO(x)^{-s-\frac{\eps}{2}}\bigl(1+|\ln \deltaO(x)|\Bigr) \le s \, C \|f\|_{L^\infty(\Omega)} \deltaO(x)^{-\eps}.
\end{align*}
The claim thus follows.
\end{proof}

\begin{proof}[Proof of Proposition~\ref{derivative-zero} (completed)]
The estimate (\ref{derivative-zero-est-1}) follows by combining equations (\ref{sec:thecases0-separation}), (\ref{eq:rn-case-second-claim}), and Lemma~\ref{separate-case}.
Moreover, (\ref{derivative-zero-est-1}) implies that $\frac{\cGs_sf-f}{s} \to - \loglap \cEO f$ almost uniformly in $\Omega$ as $s \to 0^+$, and thus 
$\cGs_s f \to f$ almost uniformly in $\Omega$ as $s \to 0^+$ by (\ref{eq:log-bound-1}).
 \end{proof}

\section{Continuity of the solution map}\label{sec:continuity}

Throughout this section, we assume that $\Omega\subset\R^N$, $N \ge 2$ is an open bounded set with $C^{2}$ boundary. We now discuss continuity results for the solution map $(s,f) \mapsto \cGs_s f$ for $s \in [0,1]$. We recall, that for $s=1$, the definition of $\cH^s_0(\Omega)$ coincides with the classical Sobolev space $H^1_0(\Omega)$. For the continuity at $s=1$ in a weak setting, we refer to \cite{BH18}. Moreover, the behavior of $s$-harmonic functions as $s\to 1^-$ has been studied in \cite{TTV18}. 

\begin{lemma}\label{lemma:continuity} Let $s_0 \in (0,1]$, $\delta \in (0,s_0)$. Moreover, let $(s_n)_n \subset (s_0-\delta,1]$ be a sequence with $s_n \to s_0$, and let $(f_n)_n \subset L^2(\Omega)$ be a sequence with $f_n \to f_0 \in L^2(\Omega)$. 
Then we have 
	\begin{align*}
		\cGs_{s_n} f_n\to \cGs_{s_0} f_0\qquad \text{ strongly in $\cH_0^{s_0-\delta}(\Omega)$ as $n \to \infty$.}
	\end{align*}
\end{lemma}

\begin{proof}
Let $u_n= \cGs_{s_n} f_n$ for $n \in \N \cup \{0\}$ and recall that $u_n \in \cH_0^{s_n}(\Omega)$ is the unique weak solution of $(-\Delta)^{s_n} u_n = f_n$ in $\Omega$, $\;u_n \equiv 0$ in $\R^N \setminus \Omega$, see Remark \ref{green-weak-solution}. Let $s':= \inf \limits_{n \in \N}s_n> s_0-\delta$. By the fractional Poincar\'e inequality, there exists $C_{s'}>0$ with 
\begin{align}\label{pi}
 \|v\|_{L^2(\R^N)}^2 \le C_{s'} \cE_{s'}(v,v) \qquad \text{for all $v \in \cH^{s'}_0(\Omega)$.}
\end{align}
Moreover, for $\eps \in (0,1]$, $n \in \N$ and $v \in \cH_0^{s_n}(\Omega)$ we have 
\begin{align*}
\cE_{s'}(v,v) &= \int_{\R^N}|\xi|^{2s'} |\hat v|^2 d\xi \le \eps^{2s'} \|v\|_{L^2(\R^N)}^2 + \int_{|\xi| \ge \eps}|\xi|^{2s'} |\hat v|^2 d\xi \\
&\le \eps^{2s'} \|v\|_{L^2(\R^N)}^2 + \eps^{2(s'-s_n)} \int_{\R^N}|\xi|^{2s_n} |\hat v|^2 d\xi \le \eps^{2s'}C_{s'}\cE_{s'}(v,v) + 
\eps^{2s'-2}\cE_{s_n}(v,v),
\end{align*}
where $\hat v$ denotes the Fourier transform of $v$.  Choosing $\eps:= \min \bigl \{1, \bigl(2C_{s'}\bigr)^{-\frac{1}{2s'}} \bigr\}$ yields  
\begin{equation}
  \label{eq:s-prime-est}
\cE_{s'}(v,v) \le C \cE_{s_n}(v,v) \qquad \text{for $v \in \cH_0^{s_n}(\Omega)$, $n \in \N$ with $C= 2 \eps^{2s'-2}$.}
\end{equation}
Moreover, by the definition of weak solution, 
\begin{align*}
\cE_{s_n}(u_n,u_n) &= \int_\Omega f_n u_n \,dx \le \|f_n\|_{L^2(\Omega)} \|u_n\|_{L^2(\Omega)} \le
C_{s'} \|f_n\|_{L^2(\Omega)} \sqrt{\cE_{s'}(u_n,u_n)} \\
&\le \sqrt{C_{s'}C} \|f_n\|_{L^2(\Omega)} \sqrt{\cE_{s_n}(u_n,u_n)}.
\end{align*}
Consequently,
\begin{equation}
  \label{eq:s-n-est}
\cE_{s'}(u_n,u_n) \le C C_{s'} \|f_n\|_{L^2(\Omega)}^2.
\end{equation}
By (\ref{eq:s-n-est}), it follows that the sequence $(u_n)_n$ is bounded in $\cH^{s'}_0(\Omega)$.
We now suppose by contradiction that, passing to a subsequence, 
 \begin{equation}
   \label{eq:contradiction-sobolev-convergence}
 \cE_{s_0-\delta}(u_{s_0}-u_n,u_{s_0}-u_n) \ge \rho>0 \qquad \text{for all $n \in \N$.}
 \end{equation}
Passing to a subsequence again and using the compactness of the embedding $\cH^{s'}_0(\Omega) \hookrightarrow \cH^{s-\delta}_0(\Omega)$, we may assume that $u_n \to u_* \in \cH^{s-\delta}_0(\Omega)$ for some $u_* \in \cH^{s-\delta}_0(\Omega)$, which also implies that
$u_n \to u_* \in L^2(\R^N)$ and therefore $\widehat{u_n} \to \widehat{u_*} \in L^2(\R^N)$.  Passing to a subsequence again, we may also assume that $\widehat{u_n} \to \widehat{u_*}$ a.e. in $\R^N$. Consequently, by Fatou's Lemma and (\ref{eq:s-n-est}),  
	\begin{align*}
		\cE_{s_0}(u_*,u_*)&=\int_{\R^N}|\xi|^{2{s_0}} |\hat u_*(\xi)|^2\ d\xi\\
		&\leq \liminf_{n\to \infty} \int_{\R^N}|\xi|^{2s_n}|\hat u_n(\xi)|^2\ d\xi =\liminf_{n\to \infty} \cE_{s_n}(u_n,u_n)\le  C C_{s'} \|f_0\|_{L^2(\Omega)}^2
	\end{align*}
	and therefore $u_*\in \cH_0^{s_0}(\Omega)$. Finally, since $C^{\infty}_c(\Omega) \subset \cH_0^{s_0}(\Omega)$ is dense and 
	\begin{align*}
		(f_0,\phi)_{L^2}=\lim_{n\to \infty}(f_n,\phi)_{L^2} = \lim_{n\to \infty}\cE_{s_n}(u_n,\phi)=\lim_{n\to \infty}(u_n,(-\Delta)^{s_n}\phi)_{L^2}&=(u_*,(-\Delta)^{s_0}\phi)_{L^2}\\
&=\cE_{s_0}(u_*,\phi)
	\end{align*}
	for all $\phi\in C^{\infty}_c(\Omega)$, we obtain that $u_*\equiv u_{s_0}$ by the uniqueness stated in Remark~\ref{green-weak-solution}. This contradicts \eqref{eq:contradiction-sobolev-convergence} and the claim is proved. 
\end{proof}

\begin{lemma}\label{lemma:continuity-smooth}
Let $s_0 \in (0,1]$, $\delta\in(0,\frac{s_0}{2})$, let $(s_n)_n\subset (s_0-\delta,\max\{1,s_0+\delta\}]$ be a sequence with $s_n \to s_0$, 
Moreover, let $(f_n)_n \subset L^\infty(\Omega)$ be a sequence with $f_n \to f_0 \in L^\infty(\Omega)$. Then we have
\[
\cGs_{s_n} f_n \to \cGs_{s_0} f_0 \quad\text{in $C^{s_0-2\delta}_0(\Omega)$ as $s\to s_0$.}
\]
\end{lemma}
\begin{proof}
Let $u_n= \cGs_{s_n} f_n$ for $n \in \N \cup \{0\}$. Then Lemma \ref{lemma:continuity} gives that $u_n\to u_{0}$ in $L^2(\Omega)$ as $n\to\infty$.
Moreover, from (\ref{omega-bar-reg}) one may deduce that $(u_n)_n$ is a bounded sequence in $C^{s_0-\delta}_0(\overline \Omega)$, where we have used that the constant in \eqref{omega-bar-reg} is independent of $s$.

Since $C^{s_0-\delta}_0(\overline \Omega)$ is compactly embedded into $C^{s_0-2\delta}_0(\overline \Omega)$ by the Arzela-Ascoli Theorem and $C^{s_0-2\delta}_0(\overline \Omega)$ is embedded in $L^2(\Omega)$, it then follows by a standard argument that $u_n \to u_{0}$ in $C^{s_0-2\delta}_0(\overline \Omega)$ as $n \to \infty$.
\end{proof}

\section{Differentiability of the solution map in \texorpdfstring{$(0,1)$}{(0,1)}}\label{sec:main-part}

Throughout this section, let $N \ge 2$, and let $\Omega \subset \R^N$ be an open and bounded set with $C^2$-boundary.
Recall the definition of the extension and restriction maps $\cEO$ and $\cRO$ as introduced in Section \ref{sec:solution-map-related}.  The aim of this section is to prove the following. 

\begin{thm}\label{thm1:differential-section}
Let $f\in C^\alpha(\overline \Omega)$ for some $\alpha>0$. Then the map 
$$
(0,1) \to L^\infty(\Omega), \qquad s \mapsto \cGs_s f
$$
is of class $C^1$ and  
\begin{equation}\label{thm1:eq2-section}
\frac{d}{ds} \cGs_s f= \cGs_s \cRO \Bigl(\loglap \cQs_s f- \loglap \cEO f \Bigr) \qquad \text{for $s \in (0,1)$.}
\end{equation}
\end{thm}

We need some preliminary estimates.

\begin{lemma}\label{gsigma:l}
Let $f \in L^\infty(\Omega)$, $s \in (0,1)$ and $\eps \in (0,\min\{s,1-s\})$. Then we have 
\begin{equation}
  \label{eq:Q-difference-1}
\lim_{\sigma \to 0}\: \bigl\| \bigl(\deltaO^{s+\eps}+ \deltaO^{N+2s-\eps}\bigr)[(\cQs_{s+\sigma}-\cQs_s)f]\bigr\|_{L^\infty(\R^N \setminus \Omega)}=0
\end{equation}
and 
\begin{equation}
  \label{eq:Q-difference-2}
\lim_{\sigma \to 0^+}\: \frac{1}{\sigma} \bigl\|\deltaO^{s+\eps} [\cFs_\sigma (\cQs_{s+\sigma}-\cQs_s)f]\bigr\|_{L^\infty(\Omega)}=0.
\end{equation}
\end{lemma}
We point out here that in (\ref{eq:Q-difference-1}) a two-sided limit is considered, whereas we may only consider a one-sided limit in (\ref{eq:Q-difference-2}).

\begin{proof} 
We start by showing (\ref{eq:Q-difference-1}). In the following, let $u_s:= \cGs_s f$ for $s \in (0,1)$. For fixed $s \in (0,1)$ and $\eps \in (0,s)$, we then have 
 \begin{equation}
\label{lemma:continuity-smooth-consequence}   
\lim_{\sigma \to 0}\|u_{s+\sigma}-u_s\|_{C^{s-\eps}(\overline{\Omega})}=0 
 \end{equation}
by Lemma \ref{lemma:continuity-smooth}, which, since $u_{s+\sigma}-u_s=0$ on $\partial\Omega$, implies that
\begin{equation}\label{hol:est}
\lim_{\sigma \to 0}\|\deltaO^{\eps-s}(\cdot)(u_{s+\sigma}-u_s)\|_{L^\infty(\Omega)}=0. 
\end{equation}
In the following, $C>0$ denotes possibly different constants depending at most on $\eps$, $\Omega$, and $s$. 
Let first $\sigma \in (0,\min\{1-s,\frac{\eps}{4}\})$. By \eqref{hol:est}, we have, for $y \in \R^N \setminus \Omega$,  
\begin{align*}
|\cQs_{s+\sigma}f(y)-&\cQs_{s}f(y)|
\leq \int_\Omega \left| \frac{c_{N,s+\sigma} u_{s+\sigma}(z)}{|y-z|^{N+2s+2\sigma}}-\frac{c_{N,s} u_{s}(z)}{|y-z|^{N+2s}}
\right|\ dz\\
&\leq |c_{N,s+\sigma}-c_{N,s}| \int_\Omega \frac{|u_{s}(z)|}{|y-z|^{N+2s}}dz\\
&+ c_{N,s+\sigma} \int_\Omega \Bigl( |u_{s+\sigma}(z)|\left||y-z|^{-N-2s-2\sigma}-|y-z|^{-N-2s}\right|+\frac{|u_{s+\sigma}(z)-u_{s}(z)|}{|y-z|^{N+2s}}
\Bigr) dz\\
&\leq \sigma\,C \int_\Omega \frac{\deltaO(z)^{s}}{|y-z|^{N+2s}}dz + 
\sigma\,C   \int_\Omega \deltaO(z)^{s+\sigma}\bigl(|y-z|^{-N-2s-2\sigma-\frac{\eps}{2}}+|y-z|^{-N-2s+\frac{\eps}{2}}\bigr)dz\\
&+\|\deltaO^{\eps-s}(\cdot)(u_{s+\sigma}-u_s)\|_{L^\infty(\Omega)} 
\int_\Omega \frac{\deltaO(z)^{s-\eps}}{|y-z|^{N+2s}}
\ dz,
\end{align*}
where, by Lemma~\ref{int:la} and using the boundedness of $\Omega$, 
\begin{align*}
\int_\Omega \frac{\deltaO(z)^{s}}{|y-z|^{N+2s}}dz &\le \frac{C m(-s,-2s,\deltaO(y))}{1+\deltaO(y)^{N+2s}} \le C \frac{1+ \deltaO(y)^{-s-\eps}}{1+\deltaO(y)^{N+2s}}\le \frac{C}{\deltaO^{s+\eps}(y)+\deltaO^{N+2s-\eps}(y)},\\
\int_\Omega \deltaO(z)^{s+\sigma}\bigl(|y-&z|^{-N-2s-2\sigma-\frac{\eps}{2}}+|y-z|^{-N-2s+\frac{\eps}{2}}\bigr)dz\\
&\le 
C  \int_\Omega \deltaO(z)^{s}\bigl(|y-z|^{-N-2s-\eps}+|y-z|^{-N-2s+\eps}\bigr)dz\\
&\le C \Bigl(\frac{m(-s,-2s-\eps,\deltaO(y))}{1+\deltaO(y)^{N+2s+\eps}} + \frac{m(-s,-2s+\eps,\deltaO(y))}{1+\deltaO(y)^{N+2s-\eps}}\Bigr) \\
&\le C \Bigl(\frac{1+ \deltaO(y)^{-s-\eps}}{1+\deltaO(y)^{N+2s+\eps}} + \frac{1+ \deltaO(y)^{-s+\eps}}{1+\deltaO(y)^{N+2s-\eps}}\Bigr)  \le \frac{C}{\deltaO^{s+\eps}(y)+\deltaO^{N+2s-\eps}(y)}
\end{align*}
and 
$$
\int_\Omega \frac{\deltaO(z)^{s-\eps}}{|y-z|^{N+2s}}
\ dz \le C m(\eps-s,-2s,\deltaO(y))\le C \frac{1+\deltaO(y)^{\eps-s}}{1+\deltaO(y)^{N+2s-\eps}} \le \frac{C}{\deltaO^{s+\eps}(y)+\deltaO^{N+2s-\eps}(y)}. 
$$
Combining these estimates with (\ref{lemma:continuity-smooth-consequence}) we deduce that 
\begin{equation}
  \label{eq:Q-difference-1-plus}
\bigl\| \bigl(\deltaO^{s+\eps}+ \deltaO^{N+2s-\eps}\bigr)[(\cQs_{s+\sigma}-\cQs_s)f]\bigr\|_{L^\infty(\R^N \setminus \Omega)}\to 0 \qquad \text{as $\sigma \to 0^+$.}
\end{equation}
Next, we let $\sigma \in (-\min \{s,\frac{\eps}{4}\},0)$. Similarly as above, we then obtain the estimate
\begin{align*}
&|\cQs_{s+\sigma}f(y)-\cQs_{s}f(y)|\\
&\leq |\sigma|\,C \int_\Omega \frac{\deltaO(z)^{s}}{|y-z|^{N+2s}}dz + 
|\sigma|\,C   \int_\Omega \deltaO(z)^{s+\sigma}\bigl(|y-z|^{-N-2s-2\sigma+\frac{\eps}{2}}+|y-z|^{-N-2s-\frac{\eps}{2}}\bigr)dz\\
&+\|\deltaO^{\eps-s}(\cdot)(u_{s+\sigma}-u_s)\|_{L^\infty(\Omega)} 
\int_\Omega \frac{\deltaO(z)^{s-\eps}}{|y-z|^{N+2s}}
\ dz\\
&= o(1) \bigl(\deltaO^{s+\eps}(y)+\deltaO^{N+2s-\eps}(y)\bigr)^{-1} + 
|\sigma|\,C   \int_\Omega \deltaO(z)^{s+\sigma}\bigl(|y-z|^{-N-2s-2\sigma+\frac{\eps}{2}}+|y-z|^{-N-2s-\frac{\eps}{2}}\bigr)dz,
\end{align*}
where,  by Lemma~\ref{int:la}, 
\begin{align*}
\int_\Omega &\deltaO(z)^{s+\sigma}\bigl(|y-z|^{-N-2s-2\sigma+\frac{\eps}{2}}+|y-z|^{-N-2s-\frac{\eps}{2}}\bigr)dz\\ 
&\le C  \int_\Omega \deltaO(z)^{s-\frac{\eps}{2}}\bigl(|y-z|^{-N-2s+\eps}+|y-z|^{-N-2s-\frac{\eps}{2}}\bigr)dz\\
&\le C \Bigl(\frac{m(\frac{\eps}{2}-s,-2s+\eps,\deltaO(y))}{1+\deltaO(y)^{N+2s-\eps}} + \frac{m(\frac{\eps}{2}-s,-2s-\frac{\eps}{2},\deltaO(y))}{1+\deltaO(y)^{N+2s+\eps}}\Bigr) \\
&\le C \Bigl(\frac{1+ \deltaO(y)^{\frac{\eps}{2}-s}}{1+\deltaO(y)^{N+2s-\eps}} + \frac{1+ \deltaO(y)^{-s-\eps}}{1+\deltaO(y)^{N+2s+\eps}}\Bigr) \le \frac{C}{\deltaO^{s+\eps}(y)+\deltaO^{N+2s-\eps}(y)}.
\end{align*}
We thus find that (\ref{eq:Q-difference-1-plus}) also holds as $\sigma \to 0^-$, and thus (\ref{eq:Q-difference-1}) follows.  To see (\ref{eq:Q-difference-2}), we note that, for $x \in \Omega$, 
\begin{align*}
\frac{1}{\sigma}& |[\cFs_\sigma (\cQs_{s+\sigma}-\cQs_s)f](x)| = \frac{\kappa_{N,\sigma}}{\sigma}\Bigl|\int_{\R^N \setminus \Omega}|x-y|^{2\sigma-N}[(\cQs_{s+\sigma}-\cQs_s)f](y)\,dy \Bigr|\\    
&\le C \bigl\| \bigl(\deltaO^{s+\eps}+ \deltaO^{N+2s-\eps}\bigr)[(\cQs_{s+\sigma}-\cQs_s)f]\bigr\|_{L^\infty(\R^N \setminus \Omega)}
\int_{\R^N \setminus \Omega} \frac{|x-y|^{2\sigma-N}}{\deltaO^{s+\eps}(y)+ \deltaO^{N+2s-\eps}(y)}\,dy,
\end{align*}
where, by Corollary~\ref{est:l-new}, applied with $a = s+ \eps$, $b = 2s-\eps$, $\lambda = 0$, and $\nu=2 \sigma-N$, 
 $$
\int_{\R^N \setminus \Omega} \frac{|x-y|^{2\sigma-N}}{\deltaO^{s+\eps}(y)+ \deltaO^{N+2s-\eps}(y)}\,dy \le 
C m(s+\eps,0,\deltaO(x)) \le \frac{C}{1-s-\eps}\Bigl(1+\frac{\deltaO^{-s-\eps}(x)}{s+\eps}\Bigr)\le C \deltaO^{-s-\eps}(x).
$$
Here we used the boundedness of $\Omega$ in the last step. Together with (\ref{eq:Q-difference-1}), these estimates yield (\ref{eq:Q-difference-2}).
\end{proof}

\begin{lemma}\label{FQ:l-prelim} Let $f \in L^\infty(\Omega)$, $s \in (0,1)$ and $\eps \in (0,\min\{s,1-s\})$. Then
$$
\lim_{\sigma \to 0^+} \bigl \|\deltaO^{s+\eps}\bigl[ \bigl(\frac{\cFs_\sigma}{\sigma} -\loglap \bigr) \cQs_{s} f\bigr] \bigr \|_{L^\infty(\Omega)}= 0.
$$
\end{lemma}

\begin{proof}
In the following, let $C>0$ denote constants depending only on $N,\Omega,s$ and $\eps.$ For $x \in \Omega$ we have, using Lemma~\ref{general-est},  
\begin{align*}
\Bigl| &\bigl[ \bigl(\frac{\cFs_\sigma}{\sigma}-\loglap \bigr) \cQs_{s} f\bigr] (x)\Bigr|\le \Bigl(\bigl|\frac{\kappa_{N,\sigma}}{\sigma}|\cdot|^{2\sigma-N} -c_N|\cdot|^{-N}\bigr| *|\cQs_{s}f|\Bigr)(x)\\
&\le \bigl|\frac{\kappa_{N,\sigma}}{\sigma}-c_N\bigr| \bigl(|\cdot|^{2\sigma-N}*|\cQs_{s}f|\bigr)(x)+
c_N \bigl(\bigl|\, |\cdot|^{2\sigma-N}-|\cdot|^{-N}\bigr| *|\cQs_{s}f|\bigr)(x)\\
&\le \sigma\, C \bigl( |\cdot|^{2\sigma-N}*|\cQs_{s}f|\bigr)(x)+
\frac{\sigma\, c_N }{2\eps}\bigl[\bigl(\, |\cdot|^{2\sigma+\eps-N}+|\cdot|^{-N-\eps}\bigr) *|\cQs_{s}f|\bigr](x),
\end{align*}
where, by (\ref{Q:est-1}) and Corollary~\ref{est:l-new}, applied with $a=s$, $b=2s$, $\lambda = 0$ and $\nu = 2\sigma-N$,  
\begin{align*}
&\bigl( |\cdot|^{2\sigma-N}*|\cQs_{s}f|\bigr)(x)\le C \|f\|_{L^{\infty}(\Omega)}\int_{\R^N \setminus \Omega}  
\frac{|x-y|^{2\sigma-N}}{\deltaO(y)^{s}+\deltaO(y)^{N+2s}}dy\\
&\le C m(s,0,\deltaO(x)) \le \frac{C}{1-s}\Bigl(1+\frac{\deltaO(x)^{-s}}{s}\Bigr) \le C\deltaO(x)^{-s}\le C \deltaO(x)^{-s-\eps}
\end{align*}
and, similarly, 
\begin{align*}
\bigl( &\bigl(\, |\cdot|^{2\sigma+\eps-N}+|\cdot|^{-N-\eps}\bigr) *|\cQs_{s}f|\bigr)(x) \le
C \|f\|_{L^{\infty}(\Omega)}\int_{\R^N \setminus \Omega}  
\frac{|x-y|^{2\sigma+\eps-N}+|x-y|^{-N-\eps}}{\deltaO(y)^{s}+\deltaO(y)^{N+2s}}dy\\   
&\le C \bigl(m(s,0,\deltaO(x))+m(s,-\eps,\deltaO(x))\bigr) \le C \bigl(\deltaO(x)^{-s}+\deltaO(x)^{-s-\eps}\bigr)
\le C \deltaO(x)^{-s-\eps}.
\end{align*}
We thus conclude that 
$$
\bigl |\bigl[ \bigl(\frac{\cFs_\sigma}{\sigma}-\loglap \bigr) \cQs_{s} f\bigr](x)\bigr|\le \sigma\, C \deltaO^{-s-\eps},
$$
and this yields the claim.
\end{proof}

\begin{remark}
\label{log-bound-remark-2}
Let $f\in L^\infty(\Omega)$. Similarly as in the proof of Lemma~\ref{gsigma:l}, we may prove that
  \begin{equation}
    \label{eq:log-bound-2}
\big|[\loglap \cQs_s f](x)\big| \leq C \|f\|_{L^\infty(\Omega)} \deltaO(x)^{-s}
\qquad \text{for $s \in (0,1)$ and $x\in \Omega$ }
\end{equation}
with some constant $C=C(N,\Omega,s)>0$. Indeed, for $x \in \Omega$ we have, by (\ref{Q:est-1}) and Corollary~\ref{est:l-new},
\begin{align*}
\big|[\loglap \cQs_s f](x)\big| \le c_N [ |\cdot|^{-N} * |\cQs_s f|](x) &\le C \|f\|_{L^\infty(\Omega)} \int_{\R^N \setminus \Omega}  
\frac{|x-y|^{-N}}{\deltaO(y)^{s}+\deltaO(y)^{N+2s}}dy\\
&\le C m(s,0,\deltaO(x)) \le C\deltaO(x)^{-s}
\end{align*}
for some constants $C=C(N,\Omega,s)>0$. Moreover, since $\cQs_sf=0$ in $\Omega$ by definition, it follows that $\loglap\cQs_s f$ is smooth in $\Omega$.
\end{remark}

\begin{cor}\label{FQ:l} Let $f \in L^\infty(\Omega)$, $s \in (0,1)$ and $\eps \in (0,\min\{s,1-s\})$. Then
$$
\lim_{\sigma \to 0^+}\|\deltaO^{s+\eps} \bigl(\frac{\cFs_\sigma}{\sigma}\cQs_{s+\sigma} f -\loglap \cQs_{s}f\bigr) \|_{L^\infty(\Omega)} =0.
$$
\end{cor}

\begin{proof}
This follows by combining (\ref{eq:Q-difference-2}) with Lemma~\ref{FQ:l-prelim}.   
\end{proof}

\begin{proof}[Proof of Theorem \ref{thm1:differential-section}]
Let $f \in C^\alpha(\overline \Omega)$ for some $\alpha>0$ and $s\in(0,1)$. We first show that 
\begin{equation}
  \label{eq:first-claim-differentiability-s}
\lim_{\sigma \to 0^+}\Bigl\| \frac{(\cGs_{s+\sigma}-\cGs_{s})f}{s} - \cDs_s f\Bigr\|_{L^\infty(\Omega)}= 0 \quad \text{with $\cDs_s f = \cGs_s \cRO \Bigl(\loglap \cQs_s f- \loglap \cEO f \Bigr)$.} 
\end{equation}
Let $\sigma\in(0,1-s)$. Using Lemma~\ref{semigroup} and \eqref{eq:G-op-definition}, we have 
\begin{align*}
\cGs_{s+\sigma}&f-\cGs_{s} f = \bigr(\cFs_{s+\sigma}-\cFs_{s}\bigr)\cEO f -\bigl(\cHs_{s+\sigma}-\cHs_s\bigr) f\\
=& \cFs_{s} \bigr(\cFs_{\sigma}-\id \bigr)\cEO f -
\cHs_s \bigr(\cRO(\cFs_{\sigma} \cEO f) -f \bigr)
+ \cHs_s \cRO(\cFs_{\sigma} \cEO f) - \cHs_{s+\sigma} f\\
=& \cFs_{s} \bigl(\bigl[\bigr(\cFs_{\sigma}-\id \bigr) \cEO f\bigr]1_\Omega\,\bigr) + \cFs_{s} \bigl(\bigl[ \cFs_{\sigma} \cEO f\bigr]1_{\R^N \setminus \Omega}\,\bigr)\\
&-\cHs_s \cRO(\bigr[\cFs_{\sigma}-\id \bigr] \cEO f\bigr)
+  \cHs_s \cRO(\cFs_{\sigma}\cEO f\bigr) - \cHs_{s+\sigma} f\\
=& \cGs_s \cRO \bigl(\bigl[\cFs_{\sigma}-\id \bigr]\cEO f\bigr)+\cFs_{s}\bigl( \bigl[\cFs_{\sigma} \cEO f\bigr]1_{\R^N \setminus \Omega}\,\bigr) +\cHs_s \cRO(\cFs_{\sigma}\cEO f\bigr) - \cHs_{s+\sigma} f,
\end{align*}
where, by Lemma~\ref{semigroup}, (\ref{eq:G-op-definition}) and (\ref{eq:Hs-identity}),  
\begin{align*}
\cHs_{s+\sigma} f &= \cFs_{s+\sigma} \cQs_{s+\sigma} f= 
\cFs_s \Bigl([\cFs_\sigma \cQs_{s+\sigma} f]1_\Omega+[\cFs_\sigma \cQs_{s+\sigma} f]1_{\R^N \setminus \Omega}\Bigr)\\
&= \cGs_s \cRO(\cFs_\sigma \cQs_{s+\sigma} f)+
 \cHs_s \cRO(\cFs_\sigma \cQs_{s+\sigma} f)+ \cFs_s \Bigl([\cFs_\sigma \cQs_{s+\sigma} f]1_{\R^N \setminus \Omega}\Bigr).
\end{align*}
Consequently, 
\begin{equation}
  \label{eq:Gs-difference-formula-0}
\cGs_{s+\sigma}f-\cGs_{s} f = \cGs_s \cRO \bigl(\bigl[\cFs_{\sigma}-\id \bigr]\cEO f -\cFs_{\sigma}\cQs_{s+\sigma} f\bigr)+ w_{\sigma}
\end{equation}
with 
$$
w_{\sigma}:=  \cFs_{s}\bigl( \bigl[\cFs_{\sigma} \cEO f\bigr]1_{\R^N \setminus \Omega}\,\bigr) +\cHs_s \cRO(\cFs_{\sigma}\cEO f\bigr)  - \cHs_s \cRO(\cFs_\sigma \cQs_{s+\sigma} f) -
\cFs_s \Bigl(\bigl(\cFs_\sigma \cQs_{s+\sigma} f\bigr)1_{\R^N \setminus \Omega}\Bigr)
$$
By Lemma~\ref{F-H-properties}, the function $w_{\sigma}$ satisfies $(-\Delta)^s w_{\sigma}= 0$ in $\Omega$, whereas, by Lemma~\ref{semigroup}, \eqref{eq:Hs-identity}, and the definition of $\cHs_s$,
\begin{align*}
w_{\sigma}&= \cFs_{s}\bigl( \bigl[\cFs_{\sigma} \cEO f\bigr]1_{\R^N \setminus \Omega}\,\bigr) +\cFs_s [(\cFs_{\sigma}\cEO f\bigr)1_\Omega] - \cFs_s [(\cFs_\sigma \cQs_{s+\sigma} f)1_\Omega] -
\cFs_s \Bigl(\bigl(\cFs_\sigma \cQs_{s+\sigma} f\bigr)1_{\R^N \setminus \Omega}\Bigr)
\\
&= \cFs_{s} \cFs_{\sigma} \cEO f  -
\cFs_s \cFs_\sigma \cQs_{s+\sigma} f= \cFs_{s+\sigma} \cEO f  -
\cHs_{s+\sigma} f = 0 \qquad \qquad \text{in $\R^N \setminus \Omega$.}
\end{align*}
We also claim that $w_{\sigma} \in L^\infty(\R^N)$. Indeed, by (\ref{eq:H-F-est}) we have 
$$
|w_{\sigma}| \le \cFs_s * g_{s,\sigma}\qquad \text{with}\quad g_{s,\sigma} = \bigl|\cFs_{\sigma} \cEO f\bigr| 
+ \bigl| \cFs_\sigma \cQs_{s+\sigma} f \bigr|,
$$
and 
$$
g_{s,\sigma} \le \cFs_{\sigma} \cEO |f| + \cFs_\sigma \cQs_{s+\sigma} |f|
$$
by Remark~\ref{Q-positivity-preserving}. Consequently, using (\ref{eq:H-F-est}) again gives
\begin{align*}
|w_{\sigma}| &\le \cFs_s *\Bigl(\cFs_{\sigma} \cEO |f| + \cFs_\sigma \cQs_{s+\sigma} |f|\Bigr) = \cFs_{s+\sigma} \cEO |f|+ \cFs_{s+\sigma} \cQs_{s+\sigma} |f|\\
&=  
\cFs_{s+\sigma} \cEO |f|+ \cHs_{s+\sigma} |f| \le 2 \cFs_{s+\sigma} \cEO |f|\le 
2\|f\|_{L^\infty(\Omega)} \cFs_{s+\sigma} 1_{\Omega}, 
\end{align*}
and this shows that $w_{\sigma} \in L^\infty(\R^N)$. We now deduce from Lemma~\ref{poisson-representation-new-zero} that $w_{\sigma} \equiv 0$, and therefore, by (\ref{eq:Gs-difference-formula-0}), 
\begin{equation}
  \label{eq:Gs-difference-formula}
\frac{\cGs_{s+\sigma}f-\cGs_{s} f}{\sigma} = \cGs_s \cRO \bigl(v_{1,\sigma} - v_{2,\sigma}\bigr) \qquad \text{with}\quad v_{1,\sigma}:= \frac{\cFs_{\sigma}-\id}{\sigma} \cEO f,\quad v_{2,\sigma} =  \frac{\cFs_{\sigma}\cQs_{s+\sigma} }{\sigma}f.
\end{equation}
By Lemma~\ref{rn-case}(ii), we have that 
\begin{equation*}
\|\Bigl(v_{1,\sigma} + \loglap \cEO f\Bigr)\deltaO^\eps(\cdot)\|_{L^\infty(\Omega)} \to 0 \qquad \text{as $\sigma \to 0^+$ for every $\eps > 0$.}    
\end{equation*}
Moreover, by Corollary~\ref{FQ:l}, 
\begin{equation}
\label{difference-formular-statement-2}\notag
\|\Bigl(v_{2,\sigma} -\loglap \cQs_{s}f\Bigr)\deltaO^{s+\eps}(\cdot)\|_{L^\infty(\Omega)} \to 0 \qquad \text{as $\sigma \to 0^+$ for every $\eps > 0$.}    
\end{equation}
Consequently, setting $v_\sigma := v_{1,\sigma}- v_{2,\sigma}$ and $v_0:= -\loglap \cEO f + \loglap \cQs_{s}f$, we have that 
\begin{equation}\label{difference-formular-statement-3}
\|(v_\sigma -v_0)\deltaO^{s+\eps}(\cdot)\|_{L^\infty(\Omega)} \to 0 \qquad \text{as $\sigma \to 0^+$ for every $\eps >0$.}    
\end{equation}
Moreover, by Lemma~\ref{k-s-eps-cor} we have, for $\eps \in (0,\min\{s,1-s\})$, 
$$
K_{s,\eps}(\Omega)= \sup_{x \in \Omega} \int_{\Omega}|x-y|^{2s-N}\deltaO(y)^{-s-\eps} dy < \infty
$$
and, by (\ref{eq:G-F-est}),
\begin{align}
\|\cGs_s \cRO (v_\sigma - v_0)\|_{L^\infty(\Omega)} &\leq \kappa_{N,s} \sup_{x \in \Omega} \int_{\Omega}|x-y|^{2s-N} |v_\sigma(y)-v_0(y)|dy\notag\\
&\le \kappa_{N,s} \|(v_\sigma -v_0)\deltaO^{s+\eps}(\cdot)\|_{L^\infty(\Omega)} K_{s,\eps}(\Omega). \notag\label{G:conv}
\end{align}
It thus follows from \eqref{difference-formular-statement-3} that $\|\cGs_s \cRO (v_\sigma - v_0)\|_{L^\infty(\Omega)} \to 0$ as $\sigma \to 0^+$, and this yields (\ref{eq:first-claim-differentiability-s}).\\
Next we show that 
\begin{equation}\label{sec:cont-deriv}
\text{the map $(0,1) \to L^\infty(\Omega)$, $s \mapsto \cDs_s f$ is continuous.} 
\end{equation}
Let $s \in (0,1)$ and $\sigma \in (-s,1-s)$. With $g_s = \cRO \loglap \Bigl( \cQs_s f- \cEO f\Bigr)$, we have
$$
[\cDs_{s+\sigma}-\cDs_{s}](f) = \Bigl(\cGs_{s+\sigma}-\cGs_{s}\Bigr)g_s + \cGs_{s+\sigma} \cRO \loglap \Bigl(\cQs_{s+\sigma} f-\cQs_{s} f\Bigr).
$$
Moreover, we have 
$$
\|\deltaO^s g_s\|_{L^\infty(\Omega)} < \infty,
$$
by (\ref{eq:log-bound-1}) and (\ref{eq:log-bound-2}). For given $\eps \in (0,\min\{s,1-s\})$, we may thus write $g_s = g_{1,s}+g_{2,s}$ with $g_{1,s} \in L^\infty(\Omega)$ and $\|\deltaO^{s+\eps} g_{2,s}\|_{L^\infty(\Omega)} <\eps$.
Consequently, if $|\sigma|< \min \{s,1-s\}-\eps$, we have 
\begin{align*}
\Bigl|& \Bigl[\Bigl(\cGs_{s+\sigma}-\cGs_{s}\Bigr)g_{2,s}\Bigr](x)\Bigr| \le \Bigl[\Bigl(\cFs_{s+\sigma}+\cFs_{s}\Bigr)|g_{2,s}|\Bigr](x)\\
& \le \|\deltaO^{s+\eps} g_{2,s}\|_{L^\infty(\Omega)}  \Bigl[\Bigl(\cFs_{s+\sigma}+\cFs_{s}\Bigr)\deltaO^{-s-\eps}\Bigr](x)\\
&\le \eps \Bigl(K_{s+\sigma,\eps}+K_{s,\eps}\Bigr) \le \eps C \bigl[( \min \{s+\sigma,1-s-\sigma\}-\eps)^{-3} +(\min \{s,1-s\}-\eps)^{-3}\bigr],
\end{align*}
by Lemma~\ref{k-s-eps-cor} and therefore 
$$
\limsup_{\sigma \to 0}\Bigl\|\Bigl(\cGs_{s+\sigma}-\cGs_{s}\Bigr)g_{2,s}\Bigr\|_{L^\infty(\Omega)} \le  \eps C (\min \{s,1-s\}-\eps)^{-3}.
$$
Moreover,
$$
\lim_{\sigma \to 0}\Bigl\|\Bigl(\cGs_{s+\sigma}-\cGs_{s}\Bigr)g_{1,s}\Bigr\|_{L^\infty(\Omega)}= 0,
$$
by Lemma~\ref{lemma:continuity-smooth}. Since $\eps \in (0,\min\{s,1-s\})$ was chosen arbitrarily, this shows that 
$$
\lim_{\sigma \to 0}\Bigl\|\Bigl(\cGs_{s+\sigma}-\cGs_{s}\Bigr)g_s\Bigr\|_{L^\infty(\Omega)}= 0.
$$
It remains to show that 
$$
\cGs_{s+\sigma}\cRO \loglap \Bigl(\cQs_{s+\sigma}f-\cQs_{s} f\Bigr) \to 0 \qquad \text{as $\sigma \to 0$.}
$$
For this we recall that (\ref{eq:G-F-est}) gives rise to the pointwise inequality 
$$
\Bigl| \cGs_{s+\sigma} \cRO \loglap \Bigl(\cQs_{s+\sigma}f-\cQs_{s} f\Bigr)\Bigr| \le  \cFs_{s+\sigma}  \Bigl|\cRO \loglap \Bigl(\cQs_{s+\sigma}-\cQs_{s} \Bigr)f \Bigr|,
$$
where, for $x \in \Omega$, by~(\ref{eq:Q-difference-1}), \eqref{eq:log-laplace2}, and the fact that $(\cQs_{s+\sigma}-\cQs_{s})f=0$ in $\Omega$,
\begin{align*}
\Bigl|&\loglap \Bigl(\cQs_{s+\sigma}-\cQs_{s} \Bigr)f\Bigr|(x)\\
&\le c_N \bigl\|(\deltaO(\cdot)^{s+\eps}+ \deltaO(\cdot)^{N+2s-\eps})(\cQs_{s+\sigma}-\cQs_s)f \bigr\|_{L^\infty(\R^N \setminus \Omega)}  \int_{\R^N \setminus \Omega}  \frac{|x-y|^{-N}}{(\deltaO(y)^{s+\eps}+ \deltaO(y)^{N+2s-\eps})}dy\\
&=o(1)\int_{\R^N \setminus \Omega}  \frac{|x-y|^{-N}}{(\deltaO(y)^{s+\eps}+ \deltaO(y)^{N+2s-\eps})}dy 
\end{align*}
and, by Corollary~\ref{est:l-new},    
$$
\int_{\R^N \setminus \Omega} \frac{|x-y|^{-N}}{(\deltaO(y)^{s+\eps}+ \deltaO(y)^{N+2s-\eps})}dy \le C m(s+\eps,0,\deltaO(x))
\le C \deltaO(x)^{-s-\eps}.
$$
Consequently, for $x \in \Omega$ we have, by Lemma~\ref{k-s-eps-cor}, 
\begin{align*}
\Bigl| \cGs_{s+\sigma} \cRO \loglap \Bigl(\cQs_{s+\sigma}-\cQs_{s} f\Bigr|(x) &\le o(1) \int_{\Omega}\deltaO^{-s-\eps}(y) |x-y|^{2(s+\sigma)-N}dy \\
&\le o(1) \int_{\Omega} \deltaO(y)^{-s-\eps} |x-y|^{2s-N}
dy \le o(1) K_{s,\eps} = o(1).
\end{align*}
We have thus proved (\ref{sec:cont-deriv}). The assertion of Theorem~\ref{thm1:differential-section} now follows by combining (\ref{eq:first-claim-differentiability-s}) and (\ref{sec:cont-deriv}) with Lemma~\ref{abstract-lemma-C-1} below. The proof is thus finished.
\end{proof}

The following Lemma -- which we have used in the proof above -- seems well known, but we could not find this precise formulation in the literature. We thus give a short proof for the convenience of the reader.

\begin{lemma}
\label{abstract-lemma-C-1}
Let $I \subset \R$ be an open interval, $E$ be a Banach space and $\alpha: I \to E$ be a curve with the following properties.
\begin{enumerate}
\item[(i)] $\alpha$ is continuous. 
\item[(ii)] $\partial_s^+ \alpha(s): = \lim\limits_{\sigma \to 0^+}\frac{\alpha(s+\sigma)-\alpha(s)}{\sigma}$ exists in $E$ for all $s \in I$. 
\item[(iii)] The map $I \to E$, $s \mapsto \partial_s^+ \alpha(s)$ is continuous. 
\end{enumerate}
Then $\alpha$ is continuously differentiable with $\partial_s \alpha= \partial_s^+ \alpha$. 
\end{lemma}

\begin{proof}[Proof]
Let $I=(a,b)$ for $a,b\in \R\cup\{\pm\infty\}$, $a<b$. We first consider the case where $\partial_s^+ \alpha \equiv 0$ on $I$ and $\alpha(t)=0$ for some $t \in I$. We then claim that 
\begin{equation}
\label{claim-standard-lemma}
\alpha \equiv 0\qquad \text{on $[t,b)$.}  
\end{equation}
To see this, we fix $\eps>0$, and we consider the set 
$$
M_\eps:= \{s \in [t,b)\::\: |\alpha(\tau)| \le \eps (\tau-t) \: \text{for $\tau \in [t,s]$}\}.
$$
Since $\alpha$ is continuous, $M_\eps$ is a relatively closed interval in $[t,b)$.  
We claim that $M_\eps= [t,b)$. Indeed, suppose by contradiction that $m:= \max M_\eps < b$. Since $\partial_s^+ \alpha(m)=0$, there exists $\delta_0 \in (0,b-m)$ with 
$$
|\alpha(m+\delta)-\alpha(m)|\le \eps \delta \qquad \text{for $\delta \in (0,\delta_0)$.}
$$
Consequently, for $\delta \in (0,\delta_0)$ we have 
$$
|\alpha(m+\delta)|\le |\alpha(m)|+ |\alpha(m+\delta)-\alpha(m)| \le \eps (m-t)+ \eps \delta =\eps(m+\delta-t)
$$
and therefore $m+\delta \in M_\eps$. This contradicts the definition of $m$, and therefore $M_\eps = [t,b)$. Since $\eps>0$ was arbitrary, (\ref{claim-standard-lemma}) is proved.

In the general case, we now fix $t \in I$ and consider the functions 
$$
\beta, \tilde \alpha: I \to E, \qquad \beta(s)=\alpha(t)+\int_t^s\partial_\tau^+\alpha(\tau)\ d\tau,\qquad \tilde  \alpha(s)= \alpha(s)-\beta(s).
$$
Since the map $s \mapsto \partial_s^+ \alpha(s)$ is continuous by assumption, the map $\beta$ is of class $C^1$ with $\partial_s \beta = \partial_s^+ \alpha$ on $I$. Consequently, $\partial_s^+ \tilde \alpha \equiv 0$ on $I$ and $\tilde \alpha(t)=0$, which by the argument above implies that 
$\tilde \alpha \equiv 0$ on $[t,b)$. Hence $\alpha$ coincides with $\beta$ and is therefore of class $C^1$ on $(t,b)$. Since $t \in I$ was chosen arbitrarily, the claim follows.
\end{proof}

\section{Nonsmooth domains}
\label{sec:appr-nonsm-doma}
The main result of this section is the following. 

\begin{prop}
\label{sec:appr-nonsm-doma-1}
Let $\Omega \subset \R^N$ be an arbitrary open bounded set. Then there exists a sequence of open bounded sets $\Omega_n \subset \Omega$, $n \in \N$ with smooth boundary and the following properties: 
\begin{enumerate}
\item[(i)] $\Omega_{n} \subset \Omega_{n+1}$ for $n \in \N$.
\item[(ii)] $\Omega= \bigcup \limits_{n \in \N} \Omega_n$. 
\end{enumerate}
Moreover, for any such sequence,
\begin{enumerate}
\item[(I)] $\lim \limits_{n \to \infty} h_{\Omega_n}(x)= h_\Omega(x)\:$ for $x \in \Omega$.
\item[(II)] If $f \in C^\alpha(\overline \Omega)$ for some $\alpha>0$, then  
\begin{equation}
  \label{eq:loglap-conv}
\loglap \cEO f(x) = \lim_{n \to \infty} \loglap \cEOn f \qquad \text{for $x \in \Omega$.}
\end{equation}
\item[(III)] If $f \in L^\infty(\Omega)$ and $s \in (0,1]$, then 
\begin{equation}
  \label{eq:green-conv}
[\cGs_s f](x)= \lim_{n \to \infty}[\cGs_s^n f](x)\quad \text{for $x \in \Omega$.}
\end{equation}
\end{enumerate}
Here $\cGs_s$ denotes the Green operator associated with the domain $\Omega$, and $\cGs_s^n$ denotes the Green operator associated with the domain $\Omega_n$ for $n \in \N$. Moreover, we identify $f$ with its restriction to $\Omega_n$.
\end{prop}

\begin{proof}
For $\eps>0$, we consider the open subsets $\Omega^\eps:= \{x \in \Omega\::\: \dist(x,\partial \Omega)> 2\eps\}$ and the characteristic functions
$$
\chi_\eps:= 1_{\Omega^\eps}
$$
We then fix a nonnegative radial function $\rho_* \in C^\infty_c(\R^N)$ with $\supp \,\rho_* \subset B_1(0)$ and $\int_{\R^N} \rho_* \,dx = 1$, and we consider the mollifying kernels 
$$
\rho_\eps \in C^\infty_c(\R^N), \qquad \rho_\eps(x)= \eps^{-N} \rho(\frac{x}{\eps}).
$$
for $\eps>0$. Moreover, we consider the functions $\eta_\eps:= \rho_\eps * \chi_\eps \in C^\infty_c(\R^N)$, which have the properties 
$$
0 \le \eta_\eps \le 1, \quad \supp \eta_\eps \subset \Omega^{\frac{\eps}{2}} \quad \text{and}\quad 
\eta_\eps \equiv 1 \quad \text{on $\Omega^\eps$.}
$$ 
By Sard's theorem, $\eta_\eps$ has a regular value $t_\eps \in (0,1)$. Consequently, the set 
$$
U^\eps:= \{x \in \R^N\::\: \eta_\eps> t_\eps\}
$$
is smoothly bounded with $\Omega^\eps \subset U^\eps \subset \Omega^{\frac{\eps}{2}}$. Choosing in particular $\eps =\eps_n:= 2^{-n}$ for $n \in \N$ and setting $\Omega_n:= U^{\eps_n}$, we have constructed a sequence of domains $\Omega_n$ satisfying 
$(i)$ and $(ii)$. Next, let $\Omega_n$, $n \in \N$, be arbitrary domains satisfying $(i)$ and $(ii)$. Moreover, let $x \in \Omega$. By monotone convergence, we then have 
$$
\lim_{n \to \infty}\int_{B_1(x)\setminus \Omega_n }\!\!\frac{1}{|x-y|^N}\ dy  
= \int_{B_1(x)\setminus \Omega}\!\!\frac{1}{|x-y|^N}\ dy
\quad \text{and}\quad 
\lim_{n \to \infty}\int_{\Omega_n \setminus B_1(x)}\!\!\frac{1}{|x-y|^{N}}\ dy
= \int_{\Omega\setminus B_1(x)}\!\!\frac{1}{|x-y|^{N}}\ dy.
$$
Consequently, $h_{\Omega_n}(x) \to h_\Omega(x)$ as $n \to \infty$, as claimed in $(I)$.  Moreover, if $f \in C^{\alpha}(\overline \Omega)$ is given for some $\alpha>0$, it follows from Lebesgue's theorem that  
$$
\lim_{n \to \infty}\int_{\Omega_n}\frac{\phi(x)-\phi(y)}{|x-y|^N}\ dy=
\int_{\Omega}\frac{\phi(x)-\phi(y)}{|x-y|^N}\ dy.
$$
Hence we deduce (\ref{eq:loglap-conv}) from $(I)$ and (\ref{eq:log-laplace2}).\\
Next we assume that $f \in L^\infty(\Omega)$. To show (\ref{eq:green-conv}), we may assume that $f$ is nonnegative, since we may split $f$ into its positive and negative part. We then set  
$u_n:= \cGs_s^n f$, $u:= \cGs_s f$. Since $v_n:=u-u_n$ satisfies 
$$
(-\Delta)^s v_n = 0\quad \text{in $\Omega_n$}, \qquad v_n \ge 0 \quad \text{in $\R^N \setminus \Omega_n$}
$$
we have $v_n \ge 0$ in $\R^N$ and therefore $0 \le u_n \le u$. In particular, it follows that   
\begin{equation}
  \label{eq:uniform-l-infty-bound}
\|u_n\|_{L^\infty(\R^N)} \le \|u\|_{L^\infty(\R^N)} < \infty \qquad \text{for all $n \in \N$.}
\end{equation}
Next, for given $x \in \Omega$, we let $n_x \in \N$, $\delta>0$ be chosen such that $B_{2\delta}(x) \subset \Omega_n$ for $n \ge n_x$. By \eqref{eq:uniform-l-infty-bound} and Lemma \ref{reg:l2}, we find that 
\begin{equation}
  \label{eq:uniform-hoelder-approx}
\sup_{n \ge n_x} \|u_n\|_{C^\alpha(B_\delta(x))} < \infty.
\end{equation}
We now argue by contradiction and assume that there exists $\eps>0$ and a subsequence, still denoted by $u_n$, such that 
$$
|u_n(x)-u(x)| \ge \eps \qquad \text{for $n \in \N$.}
$$
By (\ref{eq:uniform-hoelder-approx}), this then also implies that
\begin{equation}
  \label{eq:l-2-approx-contradiction}
\limsup_{n \to \infty} \|u_n-u\|_{L^2(B_\delta(x))}>0.
\end{equation}
Since  $u_n \in L^\infty(\Omega_n) \cap \cH_0^s(\Omega_n) \subset L^\infty(\Omega) \cap \cH_0^s(\Omega)$ and 
$$
\cE_s(u_n,u_n) = \int_{\Omega_n}f u_n dx = \int_{\Omega} f u_n dx \le \|f\|_{L^2(\Omega)} \|u_n\|_{L^2(\Omega)} \le C_{\Omega} \|f\|_{L^2(\Omega)} \sqrt{\cE(u_n,u_n)}, 
$$
by \eqref{pi}, we deduce that the sequence $u_n$ is bounded in $\cH_0^s(\Omega)$. Consequently, we may pass to a subsequence such that 
\begin{equation}
  \label{eq:l-2-convergence-contradiction}
u_n \rightharpoonup u_* \quad \text{in $\cH_0^s(\Omega)$}\qquad \text{and}\qquad
u_n \to u_* \quad \text{in $L^2(\R^N)\quad$ for $n\to \infty$.}
\end{equation}
Moreover, for any $\phi \in C^\infty_c(\Omega)$ we have $\supp \, \phi \subset \Omega_n$ for $n$ sufficiently large and therefore 
$$  
\cE_s(u_*,\phi) = \lim_{n \to \infty}\cE_s(u_n,\phi) 
= \lim_{n \to \infty}\int_{\Omega_n}f \phi\,dx = \int_{\Omega} f \phi\,dx.
$$
Hence $u_* \in \cH_0^s(\Omega)$ is a weak solution of $(-\Delta)^s u_* = f$ in $\Omega$, which, by uniqueness, implies that $u_* = u$. Consequently, $u_n \to u$ in $L^2(\R^N)$ by (\ref{eq:l-2-convergence-contradiction}), which contradicts \eqref{eq:l-2-approx-contradiction}. We thus conclude that $u_n (x) \to u(x)$ as $n \to \infty$, as claimed in (\ref{eq:green-conv}).
\end{proof}

\section{Completion of proofs}
\label{sec:compl-proofs-main}

In this section we complete the proofs of our main results given in the introduction.

\begin{proof}[Proof of Theorem~\ref{thm1:differential}]
By Theorem~\ref{thm1:differential-section}, the map $(0,1) \to L^\infty(\Omega)$, $s \mapsto u_s:= \cGs_s f$ is of class $C^1$ with $v_s:= \partial_s u_s$ given by $v_s = \cGs_s \cRO \bigl(\loglap \cQs_s f- \loglap \cEO f \bigr)$. This yields (\ref{thm1:eq2}), by Lemma \ref{G-properties}, Remark \ref{log-bound-remark-1}, and Remark \ref{log-bound-remark-2}.

Next we assume that $f \ge 0$ in $\Omega$ and $f\not\equiv 0$, which, since $\cGs_s$ is positivity preserving and $G_s$ is positive in $\Omega\times \Omega$, implies that $u_s > 0$ in $\Omega$ and therefore $w_s < 0$ in $\R^N \setminus \Omega$, where $w_s$ is defined in (\ref{eq:def-w-s}). The latter property implies, by \eqref{eq:log-laplace2}, that $-\loglap w_s < 0$ in $\Omega$ and therefore 
$$
v_s <- \cGs_s [\loglap \cEO f] \qquad \text{in $\Omega$ for $s \in (0,1)$.}
$$
If $\loglap \cEO f \ge 0$ in $\Omega$, it follows that $\cGs_s [\loglap \cEO f] \ge 0$ and therefore $v_s < 0$ in $\Omega$ for all $s \in (0,1)$. If, on the other hand, $v_s < 0$ for all $s \in (0,1)$, then $u_{s}(x) < u_{s'}(x)$ for $s, s' \in (0,1)$, $s'<s$, and $x \in \Omega$. Since $u_{s'} \to f$ almost uniformly in $\Omega$ as $s' \to 0^+$ by Proposition~\ref{derivative-zero}, we deduce that $u_s(x) < f(x)$ for all $x \in \Omega$, $s \in (0,1)$ and therefore, again by Proposition~\ref{derivative-zero},    
$$
[\loglap \cEO f](x) = -\lim_{s \to 0^+}\frac{1}{s}[u_s-f](x) \ge 0 \qquad \text{for $x \in \Omega$.}
$$   
\end{proof}

\begin{proof}[Proof of Theorem~\ref{thm1:differential-s-0}]
Setting $u_s:= \cGs_s f$ for $s \in (0,1)$, we have 
$$
  u_s \to f \quad \text{and}\quad \frac{u_s-f}{s} \to -\loglap \cEO f
\qquad \text{almost uniformly in $\Omega$ as $s\to 0^+$}
$$
by Proposition~\ref{derivative-zero}. Moreover, due to Theorem~\ref{thm1:differential} and since $f\geq 0$, we have, for every $s'\in(0,s)$ and $x \in \Omega$, that
\[
u_s(x)=u_{s'}(x)+\int_{s'}^{s}\cGs_\tau\cRO\bigl(\loglap\cQs_\tau f-\loglap \cEO f\bigr)(x)\ d\tau \leq u_{s'}(x)-\int_{s'}^s \bigl(\cGs_\tau [\loglap \cEO f]\bigr)(x)\,d\tau.
\] 
Passing to the limit $s' \to 0^+$, we find that 
$$
u_{s}(x)\le f(x) - \int_{0}^s \bigl(\cGs_\tau [\loglap \cEO f]\bigr)(x)\,d\tau \qquad \text{for $s \in (0,1)$, $x \in \Omega$,}
$$
as claimed in (\ref{eq:integral-derivative-estimate}).      
\end{proof}

\begin{proof}[Proof of Lemma~\ref{domain-inclusion}]
Let $\Omega, \Omega' \subset \R^N$ be open and bounded sets with $\Omega' \subset \Omega$. For $x \in \Omega'$, an easy calculation shows that  
\begin{equation}
\label{h-omega-identity}
h_{\Omega'}(x) - h_{\Omega}(x)= c_N
\int_{\Omega \setminus \Omega'}\frac{1}{|x-y|^{N}}\ dy
\end{equation}
and consequently, by (\ref{eq:log-laplace2}), 
\begin{align*}
[\loglap \cEOprime f](x) - [\loglap \cEO f](x)&=- c_N \int_{\Omega \setminus \Omega'}\frac{f(x)-f(y)}{|x-y|^N}\,dy + [h_{\Omega'}(x) - h_{\Omega}(x)]f(x)\\
&= c_N \int_{\Omega \setminus \Omega'}\frac{f(y)}{|x-y|^N}\,dy\ge 0.
\end{align*}
\end{proof}

\begin{proof}[Proof of Theorem~\ref{pointwise-decreasing}]
We first assume that $\Omega$ is of class $C^2$, and we fix $x \in \Omega$. Setting $u_s:= \cGs_s f$ for $s \in [0,1]$, where $\cGs_1$ denotes the classical Green operator associated to the Laplacian, we have, by Theorem~\ref{thm1:differential},
$$
u_{s_2}(x)-u_{s_1}(x)= \int_{s_1}^{s_2}v_s(x)\,ds \qquad \text{for $0<s_1<s_2<1$.}
$$
If $\loglap \cEO f \ge 0$ in $\Omega$, it follows from Theorem~\ref{thm1:differential} that $v_s \leq 0$ in $\Omega$ for all $s \in (0,1)$ and consequently
$u_{s_2}(x) \leq u_{s_1}(x)$ for $0<s_1<s_2<1$. Moreover, since $u_s \to f$ almost uniformly in $\Omega$ as $s \to 0^+$ by Proposition~\ref{derivative-zero} and $u_s \to u_1$ as $s \to 1^-$ uniformly (see Section \ref{sec:continuity}), we find that $u_s$ is in fact pointwisely decreasing in $\Omega$ with respect to $s \in [0,1]$. 
If, on the other hand, $u_s$ is pointwisely decreasing with respect to $s \in [0,1]$, it follows that $v_s \leq 0$ in $\Omega$ for all $s \in (0,1)$ and therefore $\loglap \cEO f \ge 0$ by Theorem~\ref{thm1:differential}.

Next, we assume that $\Omega \subset \R^N$ is an arbitrary bounded and open set, and we consider a sequence of open subsets $\Omega_n$, $n \in \N$ as in Proposition~\ref{sec:appr-nonsm-doma-1}. Moreover, we let $\cGs_s^n$ denotes the Green operator associated with the domain $\Omega_n$ for $n \in \N$. Since $\loglap \cEO f \ge 0$ in $\Omega$ by assumption, we also have $\loglap \cEOn f \ge 0$ in $\Omega_n$ by Lemma~\ref{domain-inclusion}. Now, fix $x \in \Omega$ and consider $n_x \in \N$ with $x \in \Omega_n$ for $n \ge n_x$. Then the maps 
$$
[0,1] \to \R, \qquad s \mapsto [\cGs^n_s f](x)
$$
are decreasing for $n \ge n_x$ by our considerations above. Since $[\cGs^n_s f](x) \to [\cGs_s f](x) =u_s(x)$ for $s \in [0,1]$ by Proposition~\ref{sec:appr-nonsm-doma-1}, it follows that $u_s(x)$ is decreasing in $s \in [0,1]$.
\end{proof}
  

\begin{proof}[Proof of Corollary~\ref{cor:application-fractional}]
As explained in the introduction, the result is an easy consequence of (\ref{eq:integral-derivative-estimate}) in the case where $\partial \Omega$ is of class $C^2$. In the case where $\Omega \subset \R^N$ is an arbitrary open and bounded set, we consider again a sequence of open subsets $\Omega_n$, $n \in \N$ as in Proposition~\ref{sec:appr-nonsm-doma-1}. We first show that 
\begin{equation}
  \label{eq:lim-h-omega}
h_0(\Omega) = \lim_{n \to \infty}h_0(\Omega_n).  
\end{equation}
To see this, we first note that, by (\ref{h-omega-identity}), we have 
$$
h_0(\Omega) \le h_0(\Omega_{n+1})\le h_0(\Omega_n) \qquad \text{for $n \in \N$}
$$
and therefore  
$$
h_0(\Omega) \le \lim_{n \to \infty}h_0(\Omega_n).
$$
Moreover, for fixed $x \in \Omega$, there exists $n_x \in \N$ with $x \in \Omega_n$ for $n \ge n_x$, which by Lemma~\ref{domain-inclusion} implies that 
$$
h_\Omega(x) = h_{\Omega_n}(x) -c_N \int_{\Omega \setminus \Omega_n}\frac{1}{|x-y|^N}\,dy \ge  h_0(\Omega_n) -c_N \int_{\Omega \setminus \Omega_n}\frac{1}{|x-y|^N}\,dy.
$$
Consequently, by monotone convergence, 
$$
h_\Omega(x) \ge \lim_{n \to \infty} h_0(\Omega_n)
$$
Since $x \in \Omega$ was chosen arbitrarily, (\ref{eq:lim-h-omega}) follows. 

Next, we let $\cGs_s^n$ denote the Green operator associated with the domain $\Omega_n$ for $n \in \N$. Since $\Omega_n$ has a $C^2$-boundary, we already know that 
$$
\|\cGs_s^n\| \leq e^{-s \bigl(h_0(\Omega_n)+\rho_N\bigr)}\qquad \text{for $n \in \N$.}
$$
Using Proposition~\ref{sec:appr-nonsm-doma-1}(III) and (\ref{eq:lim-h-omega}),
$$
\|\cGs_s\|\le  \limsup_{n \to \infty}\|\cGs_s^n\| \le 
\lim_{n \to \infty} e^{-s \bigl(h_0(\Omega_n)+\rho_N\bigr)} =e^{-s \bigl(h_0(\Omega)+\rho_N\bigr)}, 
$$
as claimed.
\end{proof}


\begin{proof}[Proof of Theorem~\ref{cor:application-fractional-relative-density}]
Let $r>0$. We recall the definition of the {\em relative $r$-density} of $\Omega$ given by 
$$
d_r(\Omega) = \sup_{x \in \Omega}d_r(x,\Omega) \qquad \text{with}\quad d_r(x,\Omega)= \frac{|B_r(x) \cap \Omega|}{|B_r|}.
$$
As remarked in the introduction, it suffices to prove the inequality~(\ref{eq:h-0-lower-bound-rel-density}). Fix $x \in \Omega$. An easy computation then shows that 
\begin{align*}
h_\Omega(x)&= c_N\int_{B_1(x)\setminus \Omega}|x-y|^{-N}\ dy-c_N\int_{\Omega\setminus B_1(x)}|x-y|^{-N}\ dy \\
&= -2 \ln r + c_N\int_{B_r(x)\setminus \Omega}|x-y|^{-N}\ dy-c_N\int_{\Omega\setminus B_r(x)}|x-y|^{-N}\ dy.
\end{align*}
Let $\Omega_x:=\Omega-x$,
\[
b(r):=|\Omega_x\cap B_r|, \qquad  t(r):=\Bigg(\frac{|\Omega|+|B_r|-b(r)}{|B_1|}\Bigg)^{\frac{1}{N}}\qquad\text{and}\qquad s(r):=\Bigg(\frac{b(r)}{|B_1|}\Bigg)^{\frac{1}{N}}.
\]
Note that $|B_{t(r)}\setminus B_r|=|\Omega_x\setminus B_r|$ and $|B_r\setminus B_{s(r)}|=|B_r\setminus \Omega_x|$. From this we deduce, similarly as in the proof of \cite[Lemma 4.11]{CW18}, that 
\begin{align}
&c_N\int_{B_r(x)\setminus \Omega}|x-y|^{-N}\ dy-c_N\int_{\Omega\setminus B_r(x)}|x-y|^{-N}\ dy\\
&= c_N\int_{B_r\setminus \Omega_x}|y|^{-N}\ dy-c_N\int_{\Omega_x\setminus B_r}|y|^{-N}\ dy\notag \geq c_N\int_{B_r\setminus B_{s(r)}}|y|^{-N}\ dy - c_N\int_{B_{t(r)}\setminus B_r}|y|^{-N}\ dy\\
&= 2\int_{s(r)}^r\rho^{-1}\ d\rho - 2\int_{r}^{t(r)}\rho^{-1}\ d\rho
=4 \ln r - 2\ln [s(r)t(r)]\label{lemma-for-cor2-estm1}
\end{align}
Consequently, 
$$
h_\Omega(x) \ge 2 \ln \frac{r}{s(r)t(r)} = \frac{2}{N} \ln \Bigl(\frac{r^N|B_1|^2}{(|\Omega|+|B_r|-b(r))b(r)}\Bigr)= - \frac{2}{N} \ln \Bigl(\bigl(\frac{|\Omega|}{|B_1|}+r^N[1-d_r(x,\Omega)]\bigr)d_r(x,\Omega)\Bigr)
$$
Now, since the function $\tau \mapsto \bigl(\frac{|\Omega|}{|B_1|}+r^N[1-\tau]\bigr)\tau$ is increasing on $[0,\tau_\sOmega]$ with $\tau_\sOmega:= \frac{1}{2}\bigl(\frac{|\Omega|}{|B_1|r^N}+1\bigr)= \frac{1}{2}\bigl(\frac{|\Omega|}{|B_r|}+1\bigr)$ and $d_r(x,\Omega) \le d_r(\Omega) \le \tau_\sOmega$ for every $x \in \Omega$, we conclude that 
$$
h_0(\Omega)= \inf_{x \in \Omega}  h_\Omega(x) \ge - \frac{2}{N} \ln \Bigl(\bigl(\frac{|\Omega|}{|B_1|}+r^N[1-d_r(\Omega)]\bigr)d_r(\Omega)\Bigr),
$$
as claimed in (\ref{eq:h-0-lower-bound-rel-density}). Hence, Theorem~\ref{cor:application-fractional-relative-density} follows from Corollary~\ref{cor:application-fractional}.
\end{proof}

\stopcontents[sections]
\appendix

\section{On uniform constants for the global \texorpdfstring{$C^s$}{Cs}-regularity}
\label{appendix}

In this section, we give proofs of the fact that the constants $C_i(\Omega)$, $i=1,2$ in \eqref{omega-bar-reg} and \eqref{sec:notation:boundary-reg} can be chosen independently of $s$. 

\subsection{Uniform constant for the boundary decay}
\label{boundary-decay-constant}

In this subsection, we consider the boundary decay estimate \eqref{sec:notation:boundary-reg}. So let $\Omega \subset \R^N$ be a bounded open set of class $C^2$, and let $g \in L^\infty(\Omega)$. Moreover, we let $s \in (0,1)$ and 
$u= \cGs_s g \in \cH^s_0(\Omega)$ be the unique weak solution of 
  \begin{equation}
    \label{eq:green-problem-appendix}
\left\{
\begin{aligned}
(-\Delta)^s u &= g &&\qquad \text{in $\Omega$,}\\
u &= 0 &&\qquad \text{in $\R^N \setminus \Omega$.}  
\end{aligned}
\right.
  \end{equation}
As in \cite[Appendix]{RS12}, we shall make use of the Kelvin transform, which, for $r>0$, defines a map
$$
K_s: \cH^s_0(B_1) \to \cH^s_0(\R^N \setminus B_{1}), \qquad 
K_sv(x)=|x|^{2s-N}v(\frac{x}{|x|^2}),
$$
see \cite[Lemma 2.2]{FW12}. Here and in the following,  for any open set 
$\Omega' \subset \R^N$, $N \ge 2$, we let $\cH^s_0(\Omega')$ denote the completion of $C^\infty_c(\Omega')$ with respect to the norm induced by the scalar product 
$$
(w,v) \mapsto \cE_s(w,v) := \int_{\R^N} |\xi|^{2s} \hat w (\xi) \hat v(\xi)\,d\xi.
$$
So far we had only used this space for bounded open sets $\Omega'$. Next we recall that, for $x_0\in \R^N$ and $r>0$, the unique weak solution $\phi \in \cH^s_0(B_1)$ of the problem 
$$
(-\Delta)^s\phi =1  \quad \text{in $B_1$,}\qquad \qquad \phi \equiv 0 \quad \text{in $\R^N \setminus B_1$}
$$
is given by $\phi(x)=\gamma_{N,s}(1-|x|^2)_+^s$ with $\gamma_{N,s}=\frac{\Gamma(\frac{N}{2})4^{-s}}{\Gamma(s+1)\Gamma(\frac{N}{2}+s)}$. A straightforward computation shows that the function 
$$
v:= K_s \phi \in \cH^s_0(\R^N \setminus B_1), \qquad v(x)= \gamma_{N,s} |x|^{-N}(|x|^{-2}-1)_+^s 
$$ 
solves
$$
(-\Delta)^s[K_s\phi ](x)=|x|^{-N-2s}\quad \text{in $\R^N\setminus B_1$.}
$$
By translation invariance and the scaling properties of $(-\Delta)^s$, it thus follows that, for $z \in \R^N$ and $r>0$, the function 
$$ 
v_{r,z} \in \cH^s_0(\R^N\setminus B_r(z)),\qquad 
v_{r,z}(x)=\gamma_{N,s}\frac{(|x-z|^{2}-r^2)_+^s}{|x-z|^{N}}
$$ 
solves 
$$
(-\Delta)^s v_{r,z}=r^{2s}|\,\cdot\,-\,z\,|^{-N-2s} \qquad \text{in $\R^N\setminus B_r(z)$.}
$$
Since $\partial \Omega$ is of class $C^2$, there exists $r \in (0,1)$ such that, at every $\theta\in \partial\Omega$, there is a unique ball $B_r(z_\theta)\subset \R^N\setminus \Omega$ with $\theta\in \partial B_r(z_\theta)$. Setting $R:=\diam(\Omega)+1$, we have $\Omega\subset B_R(z_\theta)$, and therefore
$$
(-\Delta)^s v_{r,z_\theta}=r^{2s}|\,\cdot\, -\,z_\theta\,|^{-N-2s} \geq 
c:= r^{2}R^{-N-2} \qquad\text{in $\Omega$.}
$$
Since also $v_{r,z_\theta} \ge 0$ in $\R^N$, the maximum principle implies that 
$$
|u(x)|\leq \frac{\|g\|_{L^\infty(\Omega)}}{c} v_{r,z_\theta}(x)
= \frac{\|g\|_{L^\infty(\Omega)} \gamma_{N,s}}{c}\frac{(|x-z_\theta|^{2}-r^2)^s}{|x-z_\theta|^{N}} \qquad \text{for $x \in \Omega$.}
$$
Since, moreover, for every $x \in \Omega$ there exists a point $\theta \in \partial \Omega$ with $\deltaO(x)=|x-\theta|$ and $|x-z_\theta|= \deltaO(x)+r \le R$, we conclude that 
$$
|u(x)|\leq \frac{\|g\|_{L^\infty(\Omega)} \gamma_{N,s}}{c} \frac{((\deltaO(x)+r)^{2}-r^2)^s}{
(\deltaO(x)+r)^{N}} \le \frac{\|g\|_{L^\infty(\Omega)} \gamma_{N,s}(2R)^s}{cr^N}\deltaO(x)^s \le 
\frac{2 R \|g\|_{L^\infty(\Omega)} \gamma_{N,s}}{cr^N}\deltaO(x)^s
$$
for $x \in \Omega$. Moreover, for $s\in(0,1)$ we have $\gamma_{N,s}=\frac{\Gamma(\frac{N}{2})4^{-s}}{\Gamma(s+1)\Gamma(\frac{N}{2}+s)} \leq \frac{\Gamma(\frac{N}{2})}{\Gamma_m^2}$, where $\Gamma_m$ is the minimum of the Gamma-function on $(0,\infty)$. Thus \eqref{sec:notation:boundary-reg} follows with an $s$-independent constant $C_2=C_2(\Omega)$.

\subsection{Uniform constant for the \texorpdfstring{$C^s$}{Cs}-regularity}
\label{interior-regularity-constant}

We now consider the uniform regularity estimate \eqref{omega-bar-reg}. For this, we first prove a scale invariant inequality in balls with $s$-independent constants. 

\begin{lemma}\label{reg:l2}
Let $r>0$, $s\in(0,1)$, $g\in L^{\infty}(B_r)$ and $u \in L^1(\R^N, (1+|x|)^{-N-2s}dx) \cap L^\infty(B_{r+\eps})$ for some $\eps>0$ such that $(-\Delta)^su=g$ in $B_{r}$. Then there is $C=C(N)$ with
\begin{equation}\label{eq:lemma:reg:l2}
[u]_{C^s(B_{r/2})}=\sup_{\substack{x,y\in B_{r/2}\\ x\neq y}}\frac{|u(x)-u(y)|}{|x-y|^{s}}\leq r^s C\left(\|g\|_{L^{\infty}(B_r)}+ \int_{\R^N\setminus B_{r}}\frac{\tau_{N,s}|u(z)|}{|z|^{N}(|z|^2-r^2)^s}\ dz\right),
\end{equation}
where $\tau_{N,s}=\frac{2}{\Gamma(s)\Gamma(1-s)|S^{N-1}|}$.
\end{lemma}
\begin{proof}
By rescaling the quantities in (\ref{eq:lemma:reg:l2}) and using the scaling properties of $(-\Delta)^s$, it suffices to consider the case $r=1$. We then may write $u=u_1+u_2$, where $u_1:=\cFs_sE_{\text{\tiny $B_1$}}g$ and $u_2$ solves the problem 
$$
(-\Delta)^s u_2 = 0 \quad \text{in $B_1$} \qquad u_2 \equiv v:=u -u_1 \quad \text{on $\R^N \setminus B_1$.}
$$
By Lemma~\ref{poisson-representation-new-zero}, $u_2$ has a Poisson kernel representation given by $u_2 = \int_{\R^N\setminus B_1}v(z) P_s(\cdot,z)\ dz$, where 
	\[
	P_s(x,z)=-(-\Delta)^s_z G_s(x,z)= \tau_{N,s} \frac{(1-|x|^2)^s}{(|z|^2-1)^s}|x-z|^{-N}\;\; \text{for $x\in B_1,\ z\in \R^N\setminus \overline{B_1}$
} 	\]
(see e.g. \cite{B99,B16} and the references therein). Next, let $x,y\in B_{1/2}$. We then have the estimates
\begin{align*}
&\bigl|(1-|x|^2)^s-(1-|y|^2)^s\bigr| \le 4 |x-y|, \qquad |x-z|\geq \frac{|z|}{2} \ge \frac{1}{2},  \qquad |y-z|\geq \frac{|z|}{2} \ge \frac{1}{2} \qquad \text{for $z\in \R^N\setminus B_1$,}\\
&\bigl| |x-z|^{-N}-|y-z|^{-N} \bigr| \le N |x-y| \max\{|x-z|^{-N-1},|y-z|^{-N-1}\} \le N 2^{N+1} \frac{|x-y|}{|z|^{N}} \qquad \text{for $z\in \R^N\setminus B_1$.}  
\end{align*}
Since 
$$
\frac{u_2(x)-u_2(y)}{\tau_{N,s}}= \int_{\R^N \setminus B_1}\frac{|v(z)|}{(|z|^2-1)^s} \Bigl(\frac{(1-|x|^2)^s-(1-|y|^2)^s}{|x-z|^N}
  + (1-|y|^2)^s \bigl(|x-z|^{-N}-|y-z|^{-N}\bigr)\Bigr)\ dz,
$$
we deduce that 
\begin{align}
&|u_2(x)-u_2(y)| \leq |x-y| \tau_{N,s}\bigl(2^{N+2}+N2^{N+1}\bigr) \int_{\R^N\setminus B_1} \frac{|v(z)|}{|z|^N(|z|^2-1)^s} dz \label{lipschitz-appendix}\\
&\leq  \tau_{N,s} |x-y| \bigl(2^{N+2}+N2^{N+1}\bigr) \Bigl[\int_{\R^N\setminus B_1} \frac{|u(z)|}{|z|^N(|z|^2-1)^s} dz + 
\|g\|_{L^\infty(B_1)} \int_{\R^N\setminus B_1} \frac{|[\cFs_sE_{\text{\tiny $B_1$}}1](z)|}{|z|^N(|z|^2-1)^s} dz \Bigr].\nonumber 
\end{align}
Next we note that, for $z \in B_2 \setminus B_1$,
\begin{align*}
\int_{B_1}&|\zeta-z|^{2s-N}d\ \zeta\le \int_{B_{ |z|+1} \setminus B_{|z|-1}}|y|^{2s-N}dy = |S^{N-1}|\int_{|z|-1}^{ |z|+1} t^{2s-1}dt\\
&= \frac{|S^{N-1}|}{2s}\Big((|z|+1)^{2s}-(|z|-1)^{2s}\Big)  \leq \frac{|S^{N-1}|}{2s} \Big| (|z|+1)^2-(|z|-1)^2\Big|^s\leq  \frac{2|S^{N-1}|}{s}|z|^s
\end{align*}
and therefore, using that $|z|^2-1\geq \frac{|z|^2}{2}$ for $|z|\geq 2$,
\begin{align*}
&\int_{\R^N\setminus B_1} \frac{|[\cFs_sE_{\text{\tiny $B_1$}}1](z)|}{|z|^N(|z|^2-1)^s} dz = \kappa_{N,s} \int_{\R^N\setminus B_1}\frac{1}{|z|^N(|z|^2-1)^s}\int_{B_1}|\zeta-z|^{2s-N}d\zeta d z\\
&\le \frac{2 \kappa_{N,s}|S^{N-1}|}{s} \int_{B_2 \setminus B_1}\frac{|z|^{s-N}}{(|z|^2-1)^s}d z + \kappa_{N,s}2^{N-2s}
|B_1| \int_{\R^N \setminus B_2}\frac{|z|^{2s-N}}{|z|^N(|z|^2-1)^s}d z\\
&\le  \kappa_{N,s}|S^{N-1}|^2\Bigl(\frac{2}{s} \int_{B_2 \setminus B_1}\frac{|z|^{s-N}}{(|z|-1)^s}d z + \frac{2^{N-s}}{N}
\int_{\R^N \setminus B_2}|z|^{-2N}d z\Bigr)\\
&\le  \kappa_{N,s}|S^{N-1}|^2 \Bigl(\frac{2}{s}  \int_{1}^2 t^{s-1}(t-1)^{-s}d t + \frac{2^{-s}}{N^2}\Bigr)\leq
\kappa_{N,s}|S^{N-1}|^2 \Bigl(\frac{2}{s(1-s)} + \frac{2^{-s}}{N^2}
\Bigr).
\end{align*}
Inserting this in (\ref{lipschitz-appendix}) gives 
$$
|u_2(x)-u_2(y)| \leq c_1 |x-y| \Bigl(\tau_{N,s}\int_{\R^N\setminus B_1} \frac{|u(z)|}{|z|^N(|z|^2-1)^s} dz + 
\|g\|_{L^\infty(B_1)}\Bigr)\qquad \text{for $x,y \in B_{\frac{1}{2}}$}  
$$
with 
$$
c_1 = c_1(N):= 2^{N+2}+N2^{N+1} + \sup_{s \in (0,1)} \Bigl[ \tau_{N,s} \kappa_{N,s}|S^{N-1}|^2 \Bigl(\frac{2}{s(1-s)} + \frac{2^{-s}}{N^2}
\Bigr)\Bigr]
$$
The finiteness of the supremum follows easily from the explicit asymptotics of $\tau_{N,s}$ and $\kappa_{N,s}$ as $s \to 0^+$, $s \to 1^-$.
To estimate $|u_1(x)-u_1(y)|$ for $x,y \in B_{\frac{1}{2}}$, we note the inequality 
$$
|a^{2s-N}-b^{2s-N}| \le \frac{N-2s}{N-s}|a-b|^s(a^{s-N}+b^{s-N})\qquad \text{for $a,b>0$}
$$
(see e.g. \cite[Eq. (A.3)]{FW16}), which yields
	\begin{align*}
	|u_1(x)&-u_1(y)|\leq \kappa_{N,s}\|g\|_{L^\infty(B_1)}\int_{B_1}\left||x-z|^{2s-N}-|y-z|^{2s-N}\right|\ dz\\
	&\leq \kappa_{N,s}\|g\|_{L^\infty(B_1)}\frac{N-2s}{N-s}|x-y|^s \int_{B_1}\bigl(|x-z|^{s-N}+|y-z|^{s-N}\bigr)\ dz\\
        &\leq  \|g\|_{L^\infty(B_1)}2\kappa_{N,s}\frac{N-2s}{N-s} |x-y|^s \int_{B_2}|z|^{s-N} dz\\
        & = \|g\|_{L^\infty(B_1)} \kappa_{N,s}\frac{2^{1+s}(N-2s)}{s(N-s)}|x-y|^s|S^{N-1}|
\le c_2\|g\|_{L^\infty(B_1)} |x-y|^s
\end{align*}
with $c_2= c_2(N):= \sup \limits_{s \in (0,1)} \kappa_{N,s}\frac{2^{1+s}(N-2s)}{s(N-s)}|S^{N-1}|< \infty$ (the finiteness being again a consequence of the explicit asymptotics of $\kappa_{N,s}$ as $s \to 0^+, s \to 1^-$). The claim now follows with $C:= c_1 + c_2$. 
\end{proof}

Now, to finish the proof of the uniform regularity estimate \eqref{omega-bar-reg}, we let again $\Omega \subset \R^N$ be an open bounded set of class $C^2$. It already follows from (\ref{sec:notation:boundary-reg}) that the function $u=\cGs_sg$ satisfies
$$
\|u\|_{L^\infty(\Omega)} \le C_2 [\diam \Omega]^s \|g\|_{L^\infty(\Omega)} \le C_2 \bigl(1+ \diam \Omega \bigr) \|g\|_{L^\infty(\Omega)}
$$
for $g \in L^\infty(\Omega)$, $s \in (0,1)$ with an $s$-independent constant $C_2= C_2(\Omega)$. Hence, to finish the proof, it suffices to show that
\begin{equation}\label{appendix-claim-global}
 \sup_{\substack{x,y\in \Omega \\ x\neq y}}\frac{|u(x)-u(y)|}{|x-y|^s}\leq C_1\|g\|_{L^\infty(\Omega)}.
\end{equation}
with a constant $C_1=C_1(\Omega)$. For this, let $x,y\in \overline{\Omega}$, $x\neq y$ and note that if $|x-y|<\frac{1}{2}\max\{\deltaO(x),\deltaO(y)\}$, then either $x\in B_{\deltaO(y)/2}(y)$ or $y\in B_{\deltaO(x)/2}(x)$ and thus by Lemma \ref{reg:l2} 
\[
\frac{|u(x)-u(y)|}{|x-y|^s}\leq (1+\diam(\Omega))C\left(\|g\|_{L^{\infty}(\Omega)}+\|u\|_{L^\infty(\Omega)}\right)\leq C'\|g\|_{L^{\infty}(\Omega)},
\]
where $C'=(1+\diam(\Omega))^2(C+C_2)$. If otherwise $|x-y|\geq\frac{1}{2}\max\{\deltaO(x),\deltaO(y)\}$, then by \eqref{sec:notation:boundary-reg}
\[
\frac{|u(x)-u(y)|}{|x-y|^s}\leq  \frac{2^s|u(x)|}{\deltaO(x)^s}+\frac{2^s|u(y)|}{\deltaO(y)^s}\leq 4C_2(1+\diam(\Omega))\|g\|_{L^{\infty}(\Omega)}.
\]
Hence \eqref{appendix-claim-global} holds with $C_1=\max\{C',4C_2(1+\diam(\Omega))\}$. The proof of (\ref{omega-bar-reg}) is thus finished.\\

\textbf{Acknowledgement:} The work of Sven Jarohs and Tobias Weth is supported by DAAD and BMBF (Germany) within the project 57385104. The authors would like to thank Sidy Moctar Djitte for pointing out a correction in the proof of Lemma 2.5. They would also like to thank the referee for his/her remarks.


\begin{thebibliography}{10}

\bibitem{nicola}
N.~Abatangelo.
\newblock Large {$s$}-harmonic functions and boundary blow-up solutions for the
  fractional {L}aplacian.
\newblock {\em Discrete Contin. Dyn. Syst.}, 35(12):5555--5607, 2015.

\bibitem{AJS16b}
N.~Abatangelo, S.~Jarohs, and A.~Salda\~{n}a.
\newblock Green function and {M}artin kernel for higher-order fractional
  {L}aplacians in balls.
\newblock {\em Nonlinear Anal.}, 175:173--190, 2018.

\bibitem{AJS16a}
N.~Abatangelo, S.~Jarohs, and A.~Salda\~{n}a.
\newblock On the loss of maximum principles for higher-order fractional
  {L}aplacians.
\newblock {\em Proc. Amer. Math. Soc.}, 146(11):4823--4835, 2018.

\bibitem{AJS17b}
N.~Abatangelo, S.~Jarohs, and A.~Salda\~{n}a.
\newblock Positive powers of the {L}aplacian: from hypersingular integrals to
  boundary value problems.
\newblock {\em Commun. Pure Appl. Anal.}, 17(3):899--922, 2018.

\bibitem{AS64}
M.~Abramowitz and I.A. Stegun.
\newblock {\em Handbook of mathematical functions with formulas, graphs, and
  mathematical tables}, volume~55 of {\em National Bureau of Standards Applied
  Mathematics Series}.
\newblock U.S. Government Printing Office, Washington, D.C., 1964.

\bibitem{BW20}
S.~Bartels and N.~Weber.
\newblock Parameter learning and fractional differential operators: application in image regularization and decomposition.
\newblock preprint, see arXiv:2001.03394v1, 2020.

\bibitem{BH18}
U.~Biccari and V.~Hern{\'a}ndez-Santamar{\'i}a.
\newblock {The {P}oisson equation from non-local to local}.
\newblock {\em Electron. J. Differential Equations}, Paper No. 145, 2018.

\bibitem{BGR61}
R.~M.~Blumenthal, R.~K.~Getoor, and D.~B.~Ray.
\newblock  On the distribution of first hits for the symmetric stable processes. 
\newblock {\em Trans. Amer. Math. Soc.}, 99:540–554, 1961.

\bibitem{B99}
K.~Bogdan.
\newblock Representation of {$\alpha$}-harmonic functions in {L}ipschitz
  domains.
\newblock {\em Hiroshima Math. J.}, 29(2):227--243, 1999.

\bibitem{BB99}
K.~Bogdan and T.~Byczkowski.
\newblock Potential theory for the {$\alpha$}-stable {S}chr\"odinger operator
  on bounded {L}ipschitz domains.
\newblock {\em Studia Math.}, 133(1):53--92, 1999.


\bibitem{BJK19}
K.~Bogdan, S.~Jarohs, and E.~Kania.
\newblock Semilinear dirichlet problem for the fractional laplacian.
\newblock to appear in Nonlinear Analysis (2019).


\bibitem{BKK08}
K.~Bogdan, T.~Kulczycki, and M.~Kwa\'snicki.
\newblock Estimates and structure of {$\alpha$}-harmonic functions.
\newblock {\em Probab. Theory Related Fields}, 140(3-4):345--381, 2008.

\bibitem{B16}
C.~Bucur.
\newblock Some observations on the {G}reen function for the ball in the
  fractional {L}aplace framework.
\newblock {\em Commun. Pure Appl. Anal.}, 15(2):657--699, 2016.

\bibitem{BV16}
C.~Bucur and E.~Valdinoci.
\newblock {\em Nonlocal diffusion and applications}, volume~20 of {\em Lecture Notes of the Unione Matematica Italiana}
\newblock  Springer, [Cham]; Unione Matematica Italiana, Bologna, 2016.

\bibitem{CW18}
H.~Chen and T.~Weth.
\newblock {The Dirichlet problem for the Logarithmic Laplacian}.
\newblock {\em Comm. Partial Differential Equations}, 44(11):1100--1139, 2019.


\bibitem{DG17}
 S.~Dipierro and H.-Ch.~Grunau.
\newblock Boggio’s formula for fractional polyharmonic Dirichlet problems. 
\newblock {\em Ann. Mat. Pura Appl. (4)}, 196(4):1327–1344, 2017.

\bibitem{D12}
B.~Dyda.
\newblock Fractional calculus for power functions and eigenvalues of the
  fractional {L}aplacian.
\newblock {\em Fract. Calc. Appl. Anal.}, 15(4):536--555, 2012.

\bibitem{FW12}
M.M. Fall and T.~Weth.
\newblock Nonexistence results for a class of fractional elliptic boundary
  value problems.
\newblock {\em J. Funct. Anal.}, 263(8):2205--2227, 2012.

\bibitem{FW16}
M.M. Fall and T.~Weth.
\newblock Monotonicity and nonexistence results for some fractional elliptic
  problems in the half-space.
\newblock {\em Commun. Contemp. Math.}, 18(1):1550012, 25, 2016.

\bibitem{G61}
R.~K. Getoor.
\newblock First passage times for symmetric stable processes in space.
\newblock {\em Trans. Amer. Math. Soc.}, 101:75--90, 1961.

\bibitem{G08}
L.~Grafakos.
\newblock {\em Classical {F}ourier analysis}, volume 249 of {\em Graduate Texts
  in Mathematics}.
\newblock Springer, New York, second edition, 2008.

\bibitem{G11}
P.~Grisvard.
\newblock {\em Elliptic problems in nonsmooth domains}, volume~69 of {\em
  Classics in Applied Mathematics}.
\newblock Society for Industrial and Applied Mathematics (SIAM), Philadelphia,
  PA, 2011.
\newblock Reprint of the 1985 original.

\bibitem{G15:2}
G.~Grubb.
\newblock Fractional {L}aplacians on domains, a development of {H}\"ormander's
  theory of {$\mu$}-transmission pseudodifferential operators.
\newblock {\em Adv. Math.}, 268:478--528, 2015.

\bibitem{KM17}
M.~Kassmann and A.~Mimica.
\newblock Intrinsic scaling properties for nonlocal operators.
\newblock {\em J. Eur. Math. Soc. (JEMS)}, 19(4):983--1011, 2017.

\bibitem{K97}
T.~Kulczycki.
\newblock Properties of {G}reen function of symmetric stable processes.
\newblock {\em Probab. Math. Statist.}, 17(2, Acta Univ. Wratislav. No.
  2029):339--364, 1997.

\bibitem{L72}
N.S. Landkof.
\newblock {\em Foundations of Modern Potential Theory}.
\newblock Springer-Verlag, Berlin Heidelberg New York, 1972.

\bibitem{L83}
E.H. Lieb.
\newblock Sharp constants in the {H}ardy-{L}ittlewood-{S}obolev and related
  inequalities.
\newblock {\em Ann. of Math. (2)}, 118(2):349--374, 1983.

\bibitem{LL01}
E.H. Lieb and M.~Loss.
\newblock {\em Analysis}, volume~14 of {\em Graduate Studies in Mathematics}.
\newblock American Mathematical Society, Providence, RI, second edition, 2001.

\bibitem{PV18}
B.~Pellacci and G.~Verzini.
\newblock Best dispersal strategies in spatially heterogeneous
environments: optimization of the principal eigenvalue for
indefinite fractional {N}eumann problems.
\newblock {\em J. Math. Biol.}, 76(6):1357--1386, 2018.

\bibitem{RS12}
X.~Ros-Oton and J.~Serra.
\newblock The {D}irichlet problem for the fractional {L}aplacian: regularity up
  to the boundary.
\newblock {\em J. Math. Pures Appl. (9)}, 101(3):275--302, 2014.

\bibitem{RS16}
X.~Ros-Oton and J.~Serra.
\newblock Regularity theory for general stable operators.
\newblock {\em J. Differential Equations}, 260(12):8675--8715, 2016.

\bibitem{S07}
L.~Silvestre.
\newblock Regularity of the obstacle problem for a fractional power of the
  {L}aplace operator.
\newblock {\em Comm. Pure Appl. Math.}, 60(1):67--112, 2007.

\bibitem{SV17}
S.~Sprekels and E.~Valdinoci.
\newblock A new type of identification problems: optimizing the
fractional order in a nonlocal evolution equation.
\newblock {\em SIAM J. Control Optim.}, 55(1):70--93, 2017.

\bibitem{T76}
G.~Talenti.
\newblock Elliptic equations and rearrangements.
\newblock {\em Ann. Scuola Norm. Sup. Pisa Cl. Sci. (4)}, 3(4):697--718, 1976.

\bibitem{TTV18}
S.~Terracini, G.~Tortone, and S.~Vita.
\newblock On {$s$}-harmonic functions on cones.
\newblock {\em Anal. PDE}, 11(7):1653--1691, 2018.

\end{thebibliography}

\end{document}